\documentclass[12pt]{amsart}

\usepackage{amscd,amssymb,amsfonts,amsmath,amstext,amsthm}

\newcommand{\eps}{\varepsilon}

\newcommand{\La}{\Lambda}

\def\BR{{\mathbb R}}
\def\BN{{\mathbb N}}

\def\BZ{{\mathbb Z}}
\def\BC{{\mathbb C}}
\def\CO{{\mathcal O}}

\def\D{{\mathcal D}}

\def\it{\textit}

\def\iff{\Leftrightarrow}

\newcommand{\cyc}[1]{\langle #1 \rangle}

\newcommand{\ip}[2]{\langle#1,#2\rangle}
\newcommand{\Ord}[1]{\left|\, #1 \right|}

\DeclareMathOperator{\image}{Im}  
\DeclareMathOperator{\Img}{\image}
\DeclareMathOperator{\id}{id}     
\DeclareMathOperator{\Tr}{Tr}     
\DeclareMathOperator{\class}{class}
\DeclareMathOperator{\Aut}{Aut}   

\newtheorem{thm}{Theorem}[section]
\newtheorem{df}[thm]{Definition}
\newtheorem{ex}{Example}[thm]

\newtheorem{cor}[thm]{Corollary}
\newtheorem{prop}[thm]{Proposition}
\newtheorem{lem}[thm]{Lemma}
\newtheorem{question}{Question}

\numberwithin{equation}{thm}

\def\G{\BZ_k\rtimes_n\BZ_{ql}}

\title[Higher indicators for some groups and their doubles]
{Higher indicators for some groups and their doubles}
\author{Marc Keilberg}
\address{University of Southern California, Los Angeles, CA 90089-1113}
\email{keilberg@usc.edu}

\begin{document}
\begin{abstract}
In this paper we explicitly determine all indicators for groups isomorphic to the semidirect product of two cyclic groups by an automorphism of prime order, as well as the generalized quaternion groups.  We then compute the indicators for the Drinfel'd doubles of these groups.  This first family of groups include the dihedral groups, the non-abelian groups of order $pq$, and the semidihedral groups.  We find that the indicators are all integers, with negative integers being possible in the first family only under certain specific conditions.
\end{abstract}

\maketitle


\section*{Introduction}

It was shown in \cite{LM} that the classical theorem of Frobenius and Schur, giving a
formula to compute the indicator of a simple module $V$ for a finite group $G,$ extends
to any semisimple Hopf algebra $H$. This fact was shown a bit earlier for the special case of Kac
algebras over $\BC$ in \cite{FGSV}. For such an $H$-module $V$, with character $\chi$ and indicator
$\nu(V) = \nu (\chi)$, the only possible values of $\nu (\chi)$ are $0,1$ and $-1$.
Then $\nu(V)\ne 0$ if and only if $V$ is self dual; assuming $V$ is self-dual, then $\nu(V)=+1$
if and only if $V$ admits a non-degenerate $H$-invariant symmetric bilinear form and
$\nu(V)=-1$ if and only if $V$ admits a non-degenerate $H$-invariant skew-symmetric bilinear
form.  However it is not true for Hopf algebras over $\BC$, as it was for the case of a finite group $G$, that $\nu(V)=+1$ if and only if $V$ is defined over $\BR$.

The case when $\nu(V) = +1$ for all simple $G$-modules $V$ has been of particular interest;
such groups are called \it{totally orthogonal} in \cite{GW}. This terminology seems suitable for Hopf
algebras as well, in view of the existence of the bilinear forms described above. In particular it was
known classically that if $G$ is any finite real reflection group, then all $\nu(V) = +1$ \cite{S}; this
includes the case $G = S_n$, as noted above. Some other examples are given in \cite{GW}.

Moreover, indicators of modules over Hopf algebras, in particular their higher analogues we consider herein, are proving to be a very useful invariant in the study of Hopf algebras.  For example, they have been used in classifying Hopf algebras themselves \cite{K} \cite{NS1}; in studying possible dimensions of the representations of H \cite{KSZ1}; and in studying the prime divisors of the exponent of H \cite{KSZ2} \cite{NS2}. Moreover, the indicator is invariant under equivalence of monoidal categories \cite{MaN}. Another motivation comes from conformal field theory; see the work of Bantay \cite{B1} \cite{B2}.  The notion of higher indicators has also been extended to more general categories \cite{NS1} \cite{NS2} \cite{NS3}, where quasi-Hopf algebras play a unifying role \cite{N1} \cite{N2}.  Furthermore, while it is well-known that the higher indicators for modules over groups are integers, it remains an open question if the same can be said for modules over the Drinfel'd double of a group.  Thus it is important to compute the values of the indicator for more examples.  It is the goal of this paper to provide a number of explicit formulas for the indicators of certain Hopf algebras which arise as the Drinfel'd double of group algebras.

The paper is organized in the following fashion.  In Section \ref{prelimsect}, we cover the relevant background material needed for the rest of the paper.  In Section \ref{kqsect} we consider groups of the form $\G$ (see Definition \ref{presentation}), and establish various results about their structure needed in the rest of the paper.  In Section \ref{kqrepsect} we then determine the representation theory of $\G$, and use it to calculate the higher indicators for its irreducible modules.  The indicators are contained in Theorems \ref{kq1dim} and \ref{kqnonlinear}.  We then move on to the indicators of $\D(\G)$, the Drinfel'd double of $\G$, in Section \ref{kqdoubsect}.  The indicators are contained in Theorems \ref{kqdouba} and \ref{kqdoubb}.   In the double, some interesting changes of behavior are observed, such as those detailed in Corollary \ref{negind}.  We then apply the results to give a variety of examples in Section \ref{examplesect}.  In Sections \ref{quatsect}-\ref{quatindsect}, we perform a similar set of calculations to get the indicators for the generalized quaternion groups and their doubles.  We conclude with a few questions that naturally arise.

Throughout, unless otherwise specified, $H$ will be a finite dimensional semisimple Hopf algebra
over the complex numbers $\BC$, with comultiplication $\Delta:H \to H \otimes H$ given by $\Delta(h)=\sum
h_1\otimes h_2$, counit $\epsilon:H\to \BC$ and antipode $S:H\to H$.
In particular we know from \cite{LR2} that $H$ is cosemisimple and
from \cite{LR1} that $S^2=id.$  By Maschke's theorem there exists a
unique integral $\Lambda \in \int_H$ with $\epsilon (\Lambda)=1$.  Since $\Delta$ is coassociative, we may define $\Delta^2(h) = (\Delta\otimes \id)\circ\Delta(h) = (\id\otimes\Delta)\circ\Delta(h) = \sum h_1\otimes h_2\otimes h_3$ and inductively define $\Delta^n(h) = (\Delta\otimes \id)\circ\Delta^{n-1}(h) = \sum h_1\otimes \cdots\otimes h_{n+1}$ for $n\geq 2$.  We will use $\wedge$ to denote logical conjunction, $\vee$ for logical disjunction, and $\oplus$ for exclusive disjunction in logical propositions.


\section{Preliminaries}\label{prelimsect}
\begin{df}\label{indicdef}
Given a simple $H$-module $V$ and its character $\chi$, we define the functions
$$\La^{[m]} = \sum \Lambda _1\Lambda _2\cdots\Lambda_m$$
$$\nu_m (\chi)=\chi \left( \La^{[m]}\right), m\in\BN$$
where  $\Lambda$ is the unique integral in $H$ with $\eps(\Lambda) = 1$ from the introduction.
\end{df}
The next result shows that the function $\nu_2(\chi)$ agrees with the
description of the Frobenius-Schur indicator given in the introduction:

\begin{thm} \label{LM} \cite{LM}
Let $H$ be a semisimple Hopf algebra over an algebraically closed
field $k$ of characteristic not 2. Then for $\Lambda$ and $\nu_2$ as
above, the following properties hold:
\begin{enumerate}
\item
$\nu_2 (\chi)=0,1$ or $-1$, for all simple characters $\chi.$
\item
$\nu_2 (\chi) \neq 0$ if and only if $V_{\chi} \cong V_{\chi} ^*$,
where $V_{\chi}$ is the module associated to the character.
Moreover, $\nu_2 (\chi)=1$ (respectively $-1$) if and only if $V_{\chi
}$ admits a symmetric  (respectively skew-symmetric) nondegenerate
bilinear $H$-invariant form.
\item
Considering $S\in End(H)$, $Tr(S)=\sum_{\chi} \nu_2 (\chi) \chi (1).$
\end{enumerate}
\end{thm}

Throughout, the indicator of a simple module is just the indicator of its irreducible
character.  When convenient, we will also use the representation corresponding to the associated character when denoting indicators of simple modules.

\begin{cor}\label{sum} For any semisimple Hopf algebra $H$ over an algebraically closed
field $k$ of characteristic not 2, $Tr(S)=\sum_{\chi} \chi (1)$ $\iff$ $H$ is totally orthogonal;
that is, all Frobenius-Schur indicators are 1.
\end{cor}
\begin{proof} This is clear from the last part of Theorem \ref{LM}, since the values of $\nu_2 (\chi)$
are 0,1, or -1.
\end{proof}

Note that the definition of $\nu_2$, Theorem \ref{LM}, and Corollary \ref{sum} agree with what is
known for a finite group $L$, that is for a group algebra $H = \BC L$ .

For any finite dimensional Hopf algebra $H$, there is an associated Hopf algebra $\D(H)$ known as the Drinfel'd double of $H$.  We will only need a few facts about this Hopf algebra, which we recall in this section, so we omit the full definition for the sake of brevity.  A definition for, and further details of, this Hopf algebra can be found in \cite{Mo}.  As in \cite{Mo} we will write $h\bowtie f$, $h\in H, f\in H^*$, for a typical element of $\D(H)$.  We will, however need an explicit description for the irreducible modules of $\D(G)$ when $G$ is a finite group.

\begin{prop}\label{MOrho}\cite{AF}\cite{Mo}
    Let $G$ be a finite group.  The isomorphism classes of the irreducible $\D(G)$-modules are the modules $M(\CO,\rho)$ (defined below), where $\CO=\class(s)$ is the conjugacy class of some $s\in G$ and $\rho$ is (an isomorphism class of) an irreducible representation of $C_G(s)$ on a vector space $V$.  If we enumerate $\class(s) = \{t_1, ..., t_N\}$, where by convention we take $t_1=s$, and fix $g_i\in G$ with $t_1^{g_i}=t_i$ for $1\leq i\leq N$, then we can describe the module $M(\CO,\rho)$ in the following fashion:

    As a vector space $M(\CO,\rho) = \bigoplus_{i=1}^N g_i\otimes V$, or $N$ copies of $V$ indexed by the $g_i$.  We denote an element $g_i\otimes v, v\in V$ by $g_i v$.  For the left G-module structure, we define $g.g_iv = g_j(\gamma v)$, where $g_j$ and $\gamma$ are the (necessarily unique) elements with $g g_i= g_j\gamma$ in $G$ and $g_j\in\CO, \gamma\in C_G(s)$.  Here, $\gamma$ acts on $V$ via $\rho$.  For the left $G^{*}$ structure, we specify an equivalent left $G$-comodule structure $\delta$.  Specifically, we take $\delta(g_iv) = t_i\otimes g_iv$.  In particular, $M(\CO,\rho)$ can be graded by the elements of $\CO$.
\end{prop}

We note that in the above, the specific choices of $s$ (the representative of $\CO$), the isomorphism class representative $\rho$ and corresponding vector space $V$, the enumeration of $\CO$ and the choice of the $g_i$ are not crucial.  They will all naturally yield isomorphic $\D(G)$-modules.  In particular, we are free in the subsequent to fix these choices as suits the situation, and will do so without further comment.

In order to study the indicators of $G$ and $\D(G)$, for $G$ a finite group, we need to introduce a few more pieces of notation.

\begin{df}\label{Gmsetdef}
Let $G$ be a finite group.  For any $x,y\in G$ and $m\in \BN$, define
\begin{eqnarray*}
    G_m(x) &=& \left\{ a\in G : \prod_{j=0}^{m-1} a^{-j}xa^j = 1\right\}\\
    G_m(x,y) &=& \left\{ a\in G : \prod_{j=0}^{m-1} a^{-j}x a^j = 1 \mbox{ and } x^m=y\right\}\\
    z_m(x,y) &=& |G_m(x,y)|.
\end{eqnarray*}
\end{df}
In the notation of \cite{KSZ2}, taking $F=G$ our $G_m(x,y)$ and $z_m(x,y)$ are precisely $G_{m,1}(x,y)$ and $z_{m,1}(x,y)$, hence our choice of notation.  In the notation of \cite{JM}, again taking $F=G$, our sets $G_2(x)$ are precisely the sets $F_{x,x^{-1}}$.

\begin{lem}\label{integrals}
For any finite group $G$, let $p_1$ denote the element of $(kG)^*$ dual to the identity of $G$.  Then
\begin{enumerate}
  \item The unique integral $t$ of $G$ with $\epsilon(t)=1$ is
  $$t = \frac{1}{|G|}\sum_{g\in G}g.$$
  \item The unique integral $T$ of $(kG)^*$ with $\epsilon(T)=1$ is $T = p_1$.
  \item \cite[Thm. 10.3.13]{Mo} The unique integral $\Lambda$ for $\D(G)$ such that $\eps(\Lambda)=1$ is $\Lambda=T \bowtie t$.
\end{enumerate}
\end{lem}
\begin{proof}
  The first two are easily checked and well-known.  For the proof of the third, see the given reference (which gives a more general result due to Radford).
\end{proof}

By Definition \ref{indicdef}, we now need some formulas for the the quantity
$$\Lambda^{[m]} = \sum \La_1\La_2\cdots\La_m,$$
if we wish to find the indicators for $\D(G)$.

\begin{thm}\label{sweedpow} \cite{KSZ2}
Let $G$ be a finite group and let $\La$ be integral of $\D(G)$ in Lemma \ref{integrals}.iii.  Then
\begin{eqnarray*}
    \La^{[m]} &=& \frac{1}{|G|} \sum_{g,y\in G} z_m(g,y) p_g\bowtie y\\
    &=& \frac{1}{|G|} \sum_{g\in G, \ a\in G_m(g) } p_g\bowtie a^m
\end{eqnarray*}
\end{thm}

\begin{cor}\label{indicforms} \cite{KSZ2}
Let $V=M(\class(s),\rho)$ be an irreducible $\D(G)$-module, as defined in Proposition \ref{MOrho}.  In particular, assume that $\rho$ is an irreducible representation of $C_G(s)$.  Let $\chi$ be the character of $V$, and $\eta$ the character of $\rho$.  Then
    \begin{eqnarray*}
      \nu_m(\chi) &=& \frac{1}{|C_G(s)|} \sum_{y\in C_G(s)} z_m(s,y)\eta(y)\\
      &=& \frac{1}{|G|} \sum_{g\in G, \ a\in G_m(g) } \chi(p_g\bowtie a^m)
    \end{eqnarray*}
\end{cor}

Although the first equality is much more compact, for our purposes we will find the second a little bit easier to calculate with, since we can separate the two conditions $a\in G_m(g)$ and $a^m=y$ in the definition of $G_m(g,y)$.  To this end, we explicitly state when it can happen that $\chi(p_g\otimes a^m) \neq 0$.

\begin{lem}\label{0contrib}\cite{KSZ2}
     Suppose $s,g\in G$, $m\in\BN$ and $a\in C_G(g)$, and let $\chi$ and $\eta$ be as in Corollary \ref{indicforms}.  Let $\CO_s=\class(s)$ be the conjugacy class of $s$ in $G$.  For any $y\in\CO_s$, let $y'\in G$ be such that $s^{y'}=y$.  If $g\in\CO_s$ and $a^m = \gamma^{g'}$ for some $\gamma\in C_G(s)$, then $\chi(p_g\bowtie a^m)=\eta(\gamma)$.  In all other cases, $\chi(p_g\bowtie a^m)=0$.
\end{lem}
This gives the form of $\nu_m(\chi)$ that we will use in the rest of the paper.

\begin{cor}\label{cor34}  Let $V=M(\class(s),\rho)$ be an irreducible $\D(G)$-module, as defined in Proposition \ref{MOrho}.  In particular, assume that $\rho$ is an irreducible representation of $C_G(s)$.  Let $\chi$ be the character of $V$, and $\eta$ the character of $\rho$.  Then
$$\nu_m(\chi) = \frac{1}{|G|}\sum_{g\in\CO_s, \ a\in G_m(g)} \chi(p_g\bowtie a^m)$$
\end{cor}

When we have $s\in Z(G)$, it is particularly easy to figure out the values of $\nu_m$.

\begin{prop} \label{FxF} Let $G$ be a finite group and suppose $x\in Z(G)$.  Let $V=M(x,\rho)$ be any corresponding irreducible $\D(G)$-module, as given in Proposition \ref{MOrho}.  Let $\chi$ be the character of $V$ and $\eta$ the character of $\rho$.
\begin{enumerate}
    \item Assume that $x^m=1$. Then the value $\nu_m(\chi)$ is
    exactly the same as the value of $\nu_m$ for $\rho$ (a $G$-module).  As a slight abuse of notation, we write $\nu_m(\chi) = \nu_m(\eta)$.

    \item Assume that $x^m\neq 1$. Then $\nu_m(\chi)=0$
\end{enumerate}
\end{prop}
\begin{proof}
Suppose $x^m=1$.  Then by definition we have $G_m(x) = G$.  Therefore
$$\nu_m(\chi) = \frac{1}{|G|} \sum_{a\in G_m(x)} \chi(p_x\otimes a^m) = \frac{1}{|G|} \sum_{a\in G} \eta(a^m) = \nu_m(\eta)$$

Suppose $x^m\neq 1$.  Similarly to before, we immediately conclude that $G_m(x)=\emptyset$.  Therefore $\nu_m(\chi)=0$.
\end{proof}

Note that as long as $G$ contains central elements whose order does not divide $m$, some given $m\in\BN$, then $\D(G)$ admits irreducible modules $V$ with $\nu_m(V)=0$.  In particular, if $Z(G)$ contains elements with order not dividing $2$, then $\D(G)$ admits an irreducible module $V$ with $\nu_2(V)=0$.

This completes all of the preliminary results we wished to establish.  We will now proceed to consider certain families of finite groups $G$.


\section{Groups of the form $\G$}\label{kqsect}
Our goal for the rest of the paper will be to analyze the indicators for certain classes of groups and their doubles.

\begin{df}\label{presentation}
Let $k,q,l,n\in\BN$ with $q$ a prime that divides $\Ord{\Aut(\BZ_k)}$ and $n^q\equiv 1\bmod k$ but $n\not\equiv 1\bmod k$.  For any such quadruple, we consider the non-abelian group $G=\BZ_k\rtimes_n\BZ_{ql}$ given by
\begin{eqnarray*}
G = \cyc{a, b \mid a^k=b^{ql}=1, bab^{-1} = a^n}=\G.
\end{eqnarray*}
\end{df}
When $l=1$, these groups include the dihedral groups (Section \ref{dnexample}), the non-abelian groups of order $pq$ (Section \ref{pqexample}), and the semidihedral groups (Section \ref{semiex}) amongst others.  In Section \ref{quatsect} we will consider the generalized quaternion groups, which have a similar presentation.  Our assumption that $n$ denotes a prime-order automorphism is ultimately essentially for the calculations we will perform.  Any time we wish to compute indicators involving these groups, we will be lead to consider elements of the form $(a^s b^t)^m$, and the exact nature of what these are becomes hard to predict succinctly when the order of $n$ is not prime.  See, in particular, the proofs of Theorems \ref{kqnonlinear}, \ref{kqdouba}, and \ref{kqdoubb}.

To analyze the indicators for $\G$ and its double, we will need a number of basic facts regarding the parameters $k,q,n,l$ and the structure of $\G$ itself.  We start by introducing some constants depending on $\G$ that will appear throughout our calculations.

\begin{lem}\label{minq}
Let $k,q,n,l$ be as Definition \ref{presentation}.  Define
    \begin{eqnarray}\label{cvalue}
        c&=&\gcd(n-1,k),\\\label{dvalue}
        d&=&\sum_{l=0}^{q-1}n^l = \frac{n^q-1}{n-1}.
    \end{eqnarray}
Then:
    \begin{enumerate}
      \item $n^h\equiv 1\bmod k$ $\iff$ $q\mid h$.
      \item $c\mid \gcd(n^j-1,k)$ for every $j$.
      \item $d \equiv 0 \bmod \frac{k}{c}$.
      \item $d\equiv q\bmod c$.
    \end{enumerate}
\end{lem}
\begin{proof}
    \begin{enumerate}
        \item This is by definition.  Explicitly, suppose to the contrary that $q\nmid h$.  Then $\gcd(q,h)=1$ since $q$ is a prime.  Therefore, there are $u,v\in\BZ$ with $uq+vh = 1$.  Thus, since $n\in\BZ_k^*$ by assumption,
            $$n\equiv n^{uq+vh}\equiv (n^q)^u (n^h)^u\equiv 1\bmod k,$$
            a contradiction.

        \item This follows immediately from the definition of $c$ and the factorization $n^j-1 = (n-1)\sum_{l=0}^{j-1}n^l$.

        \item We know that
            $$(n-1)\sum_{l=0}^{q-1}n^l \equiv n^q -1 \equiv 0\bmod k.$$
            Since $c=\gcd(k,n-1)$, it follows that $\sum_{l=0}^{q-1}n^l$ is in the annihilator of $c$ (viewing $\BZ_k$ as a $\BZ$-module in the usual way), hence is congruent to a multiple of $k/c$ modulo $k$.  Equivalently, $\sum_{l=0}^{q-1}n^l\equiv 0\bmod \frac{k}{c}$ as claimed.

        \item By definition of $d$, we have
            $$d-q \equiv \sum_{j=0}^{q-1}n^j - q = \sum_{j=1}^{q-1}(n^j-1) \bmod k.$$

            By the previous part, $c\mid n^j-1$ for $1\leq j\leq q-1$.  Thus every term in the above summation is a multiple of $c$.  Since $c\mid k$ by definition, $d-q$ is a multiple of $c$, which gives the claim.
    \end{enumerate}
\end{proof}

It should be noted that (iii) is the best possible zero congruence we can get, in the sense that $d\not\equiv 0\bmod k$ in general.  For example, taking $k=9$, $q=3$, and $n=7$, one finds that $d=1+7+49\equiv 3\bmod 9$.  Indeed, neither is it true that $\sum_{l=0}^{q-1}n^l\equiv k/c\bmod k$ in general.  An example of this is given by the even dihedral groups: $k\in 2\BN$, $q=2$, $n=-1$.  It should also be noted that distinct values of $n$ need not yield the same values for $c$ or $d$ if the resulting groups are not isomorphic. However, when $c=1$, we obviously have $d\equiv 0\bmod k$ by (iii).

We will see in Lemma \ref{kqcent} that $c\cdot l=|Z(G)|$, from which a stronger version of $(ii)$ follows.  The value $d$ will play an important roll in the determination of the indicators for many of the irreducible modules over $G$ and $\D(G)$ (see Theorems \ref{kqnonlinear}, \ref{kqdouba}, and \ref{kqdoubb}).

We next state a few simple identities that readily follow from the presentation of $G$.  We present them and sketch their proofs since a precise understanding of the multiplication will be useful for the calculations to come.
\begin{lem}\label{sdpident}
Let $G=\G$ be a non-abelian group as in Definition \ref{presentation}.  Then the following identities hold:
    \begin{enumerate}
      \item $b a^j b^{-1} = a^{jn}$ for all $j\in\BZ$.
      \item $b a^j = a^{jn} b$ for all $j\in\BZ$.
      \item $b^j a^i = a^{in^j}b^j$, $i,j\in\BZ$
      \item For $i,j,h\in\BZ$
      $$(a^i b^j)^h = a^{i\sum_{r=0}^{h-1}n^{jr}} b^{hj}.$$
    \end{enumerate}
\end{lem}
\begin{proof}
For (i), we have from the presentation of $G$ that
$$ba^j b^{-1} = (bab^{-1})^j = a^{jn}.$$

(ii) is clearly equivalent to (i).

For (iii), we note that (2) is the case $j=1$.  The case $j\in\BN$ then follows by a straightforward induction.  We need to establish the identity then for $0>j\in\BZ$.  To this end, we note that taking $j= n^{q-1}\equiv n^{-1}\bmod k$, we have that $b^{-1}a = a^{n^{-1}} b^{-1}$, which is the $j=-1, a=1$ case.  Another straightforward induction completes the proof.

The identity in (iv) is now another simple induction on $h$ using (iii).
\end{proof}

Our next result gives an algebraic interpretation of the constant $c$.

\begin{lem}\label{kqcent} Let $G=\G$ be as in Definition \ref{presentation}.  Let $c$ be defined as in equation (\ref{cvalue}).  Then
$$Z(G) = \cyc{a^{k/c}, b^q} \cong \BZ_c\oplus \BZ_q.$$
In particular, $c=\gcd(n^j-1,k)$ whenever $j\not\equiv 0\bmod q$.
\end{lem}
\begin{proof}
  For any $i,j,r,s\in\BZ$ we have
  \begin{eqnarray}
    (a^r b^s)a^i b^j (b^{-s} a^{-r}) &=& a^{(1-n^j)r + n^s i}b^j.
  \end{eqnarray}
  Fixing $i,j$ and letting $r,s$ be arbitrary, we see that this latter element is always equal to $a^i b^j$ if and only if
  \begin{eqnarray*}
    (n-1)s\equiv 0\bmod k \mbox{ and } n^j\equiv 1\bmod k.
  \end{eqnarray*}
  Thus we must have $j\equiv 0\bmod q$ by assumptions on $n$, and $s\equiv 0\bmod \frac{k}{c}$ by definition of $c$.  We then conclude that $Z(G) = \cyc{a^{k/c},b^q}\cong \BZ_c\oplus \BZ_l$.

  Finally, by replacing $n$ with $n^j$, $j\not\equiv 0\bmod q$, in the definition of $G$ we get an isomorphic group, from which the final claim follows.
\end{proof}

This Lemma lets us establish the next two results, which help give some algebraic meaning to the constant $d$ in equation \ref{dvalue}.

\begin{cor}\label{dchar}
Let $G=\G$ be as in Definition \ref{presentation}, and define $d$ as in equation \ref{dvalue}.  For any $K\lhd G$ and $x\in G$, let $\overline{x}$ denote the equivalency class of $x$ in $G/K$.  Then the subgroup $\cyc{a^d}$ is the smallest normal subgroup $K$ of $G$ such that $\overline{(a^i b^j)^q}\in\cyc{\overline{b^q}}$ whenever $j\not\equiv 0\bmod q$.
\end{cor}
\begin{proof}
  By Lemma \ref{sdpident}.iv, if $j\not\equiv 0\bmod q$, then
  $$(a^i b^j)^q = a^{i\sum_{r=0}^{q-1}n^{jr}} b^{qj} = a^{i\sum_{h=0}^{l-1}n^{h}} b^{qj} = a^{di}b^{qj}.$$
  Taking $i=1$, we see that $\cyc{a^d}\subseteq K$ for any $K\lhd G$ with the desired property.  By Lemmas \ref{minq}.iii and \ref{kqcent}, we have $\cyc{a^d}\subseteq Z(G)$, and thus $\cyc{a^d}\lhd G$.  This establishes all claims.
\end{proof}

\begin{prop}\label{dsplit}
Let $G\cong\G$ be as in Definition \ref{presentation}.  Let $d_G$ be defined as in equation (\ref{dvalue}).  Also define $h=\gcd(d_G,k)$.
\begin{enumerate}
  \item If $\gcd(c,k/c)=1$, then
    $$G\cong \BZ_{c}\oplus\left(\BZ_{k/c}\rtimes_n\BZ_{ql}\right).$$
    If also $q\nmid l$, then we have the additional isomorphism
    $$G\cong Z(G)\oplus\left(\BZ_{k/c}\rtimes_n \BZ_q\right) \cong \BZ_c\oplus \BZ_l \oplus \left(\BZ_{k/c}\rtimes_n \BZ_q\right).$$
    In particular, $Z(\BZ_{k/c}\rtimes_n\BZ_{ql})=\cyc{b^q}$ and $Z(\BZ_{k/c}\rtimes_n \BZ_q)=1$.
  \item
    If $q\nmid \frac{k}{h}$, then $\cyc{a^{d_G}}$ is a retract and direct summand of $G$. In particular,
    $$G\cong \BZ_{k/h}\oplus\left(\BZ_h\rtimes_n\BZ_{ql}\right),$$
    where $H=\BZ_h\rtimes_n\BZ_{ql}$ is non-abelian and has $d_H\equiv d_G\equiv 0\bmod h$.
\end{enumerate}
\end{prop}
\begin{proof}
  The final claims in (i) and (ii) are immediate consequences of the definitions of $d$ and $h$.

  (i) Suppose that $\gcd(c,\frac{k}{c})=1$.  Let $u\in\BZ$ be such that $uc\equiv 1\bmod \frac{k}{c}$.  Set $K=\cyc{a^{k/c}}$, a normal subgroup of $G$ by Lemma \ref{kqcent}.  Then using the presentation of $G$, we can give a presentation for $G/K$ as
  $$G/K \cong \cyc{x,y \ | \ x^{k/c}=y^{ql}=1, yxy^{-1}=x^n} \cong \BZ_{k/c}\rtimes_n\BZ_{ql}.$$
  Let $\pi\colon G\to G/K$ be the canonical projection.  Define $\xi \colon G/K\to G$ by $\xi(x) = a^{cu}$ and $\xi(y)=b$.  It is then easily verified that $\xi$ is a group homomorphism.  Moreover, by assumptions on $c$ and $u$, $\xi$ also satisfies $\pi\circ\xi = \id_{G/K}$.  Thus $K$ has a complement in $G$.  Since $K\subseteq Z(G)$, it follows that $K$ is in fact a direct summand (and therefore retract) of $G$, and the first isomorphism in (i) follows.

  Suppose now that we also had $q\nmid l$.  Let $v\in\BZ$ be such that $vl\equiv 1\bmod q$.  Let $H=Z(G)$, an obviously normal subgroup of $G$.  Then using the presentation of $G$, we can give a presentation for $G/H$ as
  $$G/H \cong \cyc{x,y \ | \ x^{k/c}=y^q = 1, yxy^{-1}=x^n} \cong \BZ_{k/c}\rtimes_n \BZ_q.$$
  Again we let $\pi\colon G\to G/H$ be the canonical projection.  Define $\psi\colon G/H\to G$ by $\psi(x)=a^{cu}$ and $\psi(y)=b^{vl}$.  It is then easily verified that $\psi$ is a group homomorphism such that $\pi\circ\psi=\id_{G/H}$.  Thus $Z(G)$ has a complement in $G$, and is thus necessarily a direct summand of $G$, giving the second isomorphism.  The final isomorphism follows directly from Lemma \ref{kqcent}.

  (ii)  Suppose that $q\nmid \frac{k}{h}$.  Let $u\in\BZ$ be such that $uq\equiv 1\bmod \frac{k}{h}$.  Let $K=\cyc{a^{d_G}}=\cyc{a^h}\cong\BZ_{k/h}$.  Since conjugation by $b$ in $G$ fixes $Z(G)$, by definition of $d_G$ we have
  $$a^{id_G}=a^{qi}, \forall a^i\in Z(G).$$
  So define $\pi\colon G\to G$ by $\pi(a^i b^j)=\pi(a^i)\pi(b^j)=a^{udi}$.  Then it easily follows that $\pi$ is a group homomorphism such that $K=\Img(\pi)$ and $\pi|_{K}=\id_K$.  Setting $H=\ker(\pi)$, a necessarily normal subgroup of $G$, then by \cite[Thm. 7.20.iv]{R} $G\cong H\rtimes K$.  Since $K$ is also normal in $G$, we in fact have $G\cong H\times K$.  By definition of $\pi$, it easily follows that $H=\ker(\pi)=\cyc{a^{k/h},b}=\cyc{x,y \ | \ x^h = y^{ql}=1, yxy^{-1}=x^n} = \BZ_h\rtimes_n\BZ_{ql}$.  This establishes the desired isomorphism.
\end{proof}
\begin{ex}\label{splitex} Here are a few examples and counterexamples of the previous Proposition.
\begin{enumerate}
  \item If $\gcd(h,k/h)=1$, then part (ii) of the proposition applies.  This is obvious when $q\nmid k$, and when $q\mid k$ it follows from Corollary \ref{classcor}.
  \item Taking $k=8$, $q=2$, $n=3$ gives $c=2$, $d=4$, $h=4$.  Thus $\gcd(c,k/c)=\gcd(2,2)=2$ and $\gcd(q,k/h)=\gcd(2,2)=2$, so neither part of the proposition applies.  Another example of this with $q\neq 2$ is given by $k=99$, $q=3$, $n=34$.
  \item Taking $k=12$, $q=2$, $n=5$ gives $c=4$, $d=6$, $h=6$.  Thus $\gcd(c,k/c)=\gcd(4,3)=1$ and $\gcd(q,k/h)=\gcd(2,2)=2$.  Therefore part (i) of the proposition applies but part (ii) does not.  Another example of this with $q\neq 2$ is given by $k=603$, $q=3$, and $n=37$.
  \item Taking $k=12$, $q=2$, $n=7$ gives $c=6$, $d=8$, and $h=4$.  Thus $\gcd(c,k/c)=\gcd(6,2)=2$ and $\gcd(q,k/h)=\gcd(2,3)=1$.  Therefore part (i) of the proposition does not apply but part (ii) does.  No examples of this behavior with $2<q<3000$ and $k\leq 200,000$ exist, suggesting that if $q>2$ and $\gcd(q,k/h)=1$, then $\gcd(c,k/c)=1$.  Proving this would seem to require some analysis of cyclotomic polynomials in order to better understand the values $d$ and $h$.
  \item Taking $k=33$, $q=2$, $n=10$ gives $c=3$, $d=11$, and $h=11$.  Thus $\gcd(c,k/c)=\gcd(3,11)=1$ and $\gcd(q,k/h)=\gcd(2,3)=1$.  Therefore both parts of the proposition apply.  Another example of this with $q\neq 2$ is given by $k=7$, $q=3$, and $n=2$.
\end{enumerate}
\end{ex}

Now that we know the centers of our groups, we can also give a complete description of all (non-singleton) conjugacy classes.  We will need this when considering the representation theory of $\G$, as well as to determine the irreducibles modules over $\D(\G)$ (see Proposition \ref{MOrho}).

\begin{lem} \label{kqclass} Let $G\cong\G$ be as in Definition \ref{presentation}, define $c$ as in equation \ref{cvalue}, and let $i,j\in\BZ$.
  \begin{enumerate}
    \item If $a^i\not\in Z(G)$, and $q\mid j$, then $\class(a^ib^j) = \{ a^ib^j, a^{in}b^j, a^{in^2}b^j,...,a^{in^{q-1}}b^j\}$ and $|\class(a^i b^j)| = q$.  In particular, we have that $k\equiv c\bmod q$, and there are a total of $l(k-c)/q$ distinct conjugacy classes of this form.
    \item If $q\nmid j$, then $$\class(a^i b^j) = \{ a^{i + hc}b^j\}_{h=1}^{k/c},$$
    In particular, $|\class(a^i b^j)| = k/c$, and there are $cl(q-1)$ distinct conjugacy classes of this form.
    \item All conjugacy classes in $G$ are either singletons or one of the above.
  \end{enumerate}
\end{lem}
\begin{proof}
  iii) is immediate..

  i) Suppose $q\nmid j$ and $a^i\not\in Z(G)$.  It is clear that all conjugates of $a^i$ are given as $b^j a^i b^{-j}=a^{in^j}$ for some $j$.  When $a^i\not\in Z(G)$ these elements are all distinct for $1\leq j\leq q-1$.  Thus the first part follows.  For the second part, we observe that there are $k-c$ non-central powers of $a$.  Since the conjugacy class of each such power of $a$ has order $q$ and $\bigcup_{h=1}^{p} \class(a^h) = \cyc{a}$, we subsequently must have that $k-c\equiv 0\bmod q$ and that there are $(k-c)/q$ distinct conjugacy classes of non-central powers of $a$.  There are then $l$ distinct choices of $b^j$ with $q\mid j$, giving a total of $l(k-c)/q$ conjugacy classes of the relevant form.

  ii)  Suppose $q\nmid j$.  We first have that
    \begin{eqnarray*}
        (a^s b^t) (a^i b^j) (b^{-t} a^{-s}) &=& a^s (b^t a^i b^{-t})b^j a^{-s}\\
        &=& a^s a^{i n^t} b^j a^{-s}\\
        &=& a^{s-s n^j + i n^t} b^j\\
        &=& a^{-s(n^j-1)+ i n^t} b^j.
    \end{eqnarray*}
    Taking $t=j$, the exponent on $a$ is given by $i+(i-s)(n^j-1)$.  By Lemma \ref{kqcent}, $\gcd(n^j-1,k)=c$, and since $s$ can be chosen arbitrarily, this value may be written as $i+ch$, $0\leq h<k/c$ and vice versa.  Thus $|\class(a^i b^j)|\geq k/c$.

    Now for a generic $t$, using Lemma \ref{minq} we can write the exponent on $a$ as
    \begin{eqnarray*}
      i n^t-s(n^j-1) = i + i(n^t-1) - s(n^j-1) = i+ch,& h\in\BZ.
    \end{eqnarray*}
    Therefore, $|\class(a^i b^j)|=k/c$ and $\class(a^ib^j) = \{ a^ib^j, a^{in}b^j, a^{in^2}b^j,...,a^{in^{q-1}}b^j\}$ as desired.

    For the remaining claim, if $m$ is the number of distinct conjugacy classes amongst the $\class(a^i b^j)$, then by the class equation
    $$kql= cl + q\frac{l(k-c)}{q}+m\frac{k}{c},$$
    from which it follows that $m=cl(q-1)$ as desired.
\end{proof}
This result gives us a couple of simple corollaries that will be useful later in establishing formulas for indicators.
\begin{cor}\label{classcor}
Let $k,q,n,d$ be as in Definition \ref{presentation} and equation (\ref{dvalue}).  Then
\begin{enumerate}
\item $q\mid k \iff q\mid c$
\item $q\mid k \Rightarrow q\mid d$
\item $q\mid d \iff n\equiv 1\bmod q$.
\end{enumerate}
\end{cor}
\begin{proof}
  By the preceding lemma we have that $k\equiv c\bmod q$, so (i) follows.  Since we also have that $d\equiv q\bmod c$ by Lemma \ref{minq}, it follows that if $q\mid c$, then $q\mid d$, which gives (ii).

  The equivalence in (iii) is an immediate consequence of \cite[Theorem 2.4.3]{W} and the definition of $d$.
\end{proof}
In general, it fails to be true that $k\equiv d\bmod q$.  A simple example is given by taking $k=15, q=2,n=11$, which gives $d=12\not\equiv 15\bmod 2$.

\begin{cor}\label{classcor2}
Let $k,q,n,d$ be as in Definition \ref{presentation} and equation (\ref{dvalue}).  Let $m,r\in\BN$.  Then
$$q\mid m \Rightarrow \left( k\mid mr \Rightarrow kq\mid mdr\right).$$
\end{cor}
\begin{proof}
  Suppose $q\mid m$ and $k\mid mr$.  If $q\nmid k$, then since $q\mid m$ we have $kq\mid mr$, and thus $kq\mid mdr$.  On the other hand, if $q\mid k$, then $q\mid d$ by the previous corollary, and again we conclude that $kq\mid mdr$.
\end{proof}

\begin{lem}\label{classcor3}
Let $k,q,n$ be as in Definition \ref{presentation}.  Define $d$ as in equation (\ref{dvalue}).  Suppose that $q>2$.  Then the following hold:
\begin{enumerate}
    \item $\sum_{j=0}^{q-2}(j+1)n^j \equiv -d \bmod q$
    \item If $q\mid d$, then $kq\mid d(d-q)$
    \item $q^2\nmid d$.
\end{enumerate}
\end{lem}
\begin{proof}
  We start by observing that
  \begin{eqnarray}
    d - q &=& \sum_{j=0}^{q-1}n^j -q\nonumber\\
    &=& \sum_{j=1}^{q-1}n^j - (q-1)\nonumber\\
    &=& \sum_{j=1}^{q-1}(q-j)n^j - \sum_{j=0}^{q-2}(q-1-j)n^j\nonumber\\
    &=& \sum_{j=0}^{q-2}(q-1-j)n^{j+1} - \sum_{j=0}^{q-2} (q-1-j)n^j\nonumber\\
    &=& (n-1)\left(\sum_{j=0}^{q-2}(q-1-j)n^j\right).\label{dmqfact}
  \end{eqnarray}
  Therefore,
  \begin{eqnarray*}
    d(d-q) = (n^q-1)\left(\sum_{j=0}^{q-2}(q-1-j)n^j\right).
  \end{eqnarray*}
  By assumption on $n$, we have that $n^q\equiv 1\bmod k$.  Furthermore,
  $$\sum_{j=0}^{q-2} (q-1-j)n^j\equiv -\sum_{j=0}^{q-2} (j+1)n^j\bmod q.$$
  We claim that $\sum_{j=0}^{q-2} (j+1)n^j\equiv -d\bmod q$.  This will prove (i) and (ii).

  To this end, define
  \begin{eqnarray}
    f(n) = \sum_{j=0}^{q-2}n^{j+1} = \frac{n^q-1}{n-1}-1.
  \end{eqnarray}
  Differentiating all sides with respect to $n$, we get
  $$f'(n) = \sum_{j=0}^{q-2}(j+1)n^j = q n^{q-1} - \frac{n^q-1}{n-1} = q n^{q-1} - d,$$
  from which it immediately follows that
  \begin{eqnarray}\label{dmqequiv}
    \sum_{j=0}^{q-2}(j+1)n^j \equiv -d \bmod q,
  \end{eqnarray}
  which is the congruence desired.

  Finally, for (iii), we first observe that if $q\nmid d$ then clearly $q^2\nmid d$.  So suppose that $q\mid d$.  By Corollary \ref{classcor}.iii, we then have $n\equiv 1\bmod q$.  Therefore, by (\ref{dmqfact}) and (\ref{dmqequiv}) we conclude that $q^2\mid d-q$.  Thus $q^2\nmid d$.
\end{proof}
Note that the assumption $q>2$ in the above is essential.  For example, taking $k=8, q=2,n=3$ gives an easy counter-example to part (iii).  Indeed, for $k=2^s$, $q=2$, and $n=2^{s-1}-1$, where $s\geq 3$, then $2^{s-1}$ divides $d$.

We conclude this section by determining the nature of the centralizers in $G$.

\begin{lem}\label{kqnormalize} Let $G=\BZ_k\rtimes_n\BZ_q$ be as in Definition \ref{presentation}.  Let $i,j\in\BN$ and define $c$ as in equation \ref{cvalue}.  Then
\begin{enumerate}
  \item $C_G(g) = G$ $\forall g\in Z(G)$.
  \item If $q\mid j$ and $a^i\not\in Z(G)$, then $C_G(a^ib^j) = \cyc{a,b^q} \cong \BZ_k\oplus\BZ_l$.
  \item Suppose $j\not\equiv 0\bmod q$.  Then $C_G(a^i b^j) = \cyc{a^i b^j, a^{k/c},b^q}$ is an abelian group and $|C_G(a^i b^j)|=clq$.
\end{enumerate}
\end{lem}
\begin{proof}
  i) Is trivial.

  ii) Follows immediately from the definition of $G$.

  iii) Since $Z(G) = \cyc{a^{k/c},b^q}$, we immediately have $C_G(a^i b^j)\supseteq \cyc{a^i b^j, a^{k/c},b^q}$.

  Now suppose $a^s b^t \in C_G(a^i b^j)$.  Then
  $$(a^s b^t) (a^i b^j) (b^{-t} a^{-s}) = a^{(1-n^j)s + i n^t} b^j = a^i b^j$$
  if and only if $i \equiv (1-n^j)s + in^t \bmod k$ $\iff$ $(n^j-1) s \equiv (n^t-1)i \bmod k$.  If $t\equiv 0\bmod q$, this implies $(n^j-1)s\equiv 0\bmod k\iff s\equiv 0\bmod k/c$, whence $a^s b^t \in Z(G)$.
  On the other hand, suppose $t\not\equiv 0\bmod q$.  Then by Lemma \ref{kqcent}, the fraction $(n^t-1)/(n^j-1)$ corresponds to a well-defined unit $u_{j,t}$ modulo $k/c$, and so $s\equiv u_{j,t}i\bmod k/c$.  Let $r\in\BN$ be such that $jr\equiv t\bmod q$.  Then
  \begin{eqnarray*}
    (a^i b^j)^r = a^{i\sum_{l=0}^{r-1}n^{jl}}b^{jr}.
  \end{eqnarray*}
  Now by geometric series, we have that $i\sum_{l=0}^{r-1}n^{jl} \equiv i u_{j,t}\bmod k/c$.  We conclude that $a^s b^t = z(a^i b^j)^r$, for some $z\in Z(G)$.  Therefore $C_G(a^i b^j) = \cyc{a^i b^j, a^{k/c},b^q}$ as claimed.  That $C_G(a^i b^j)$ is abelian is immediate.  Furthermore, $|C_G(a^i b^j)| = |G|/|\class(a^i b^j)| = cql$ by Lemma \ref{kqclass}.  This establishes all claims.
\end{proof}

We should note that the subgroups $C_G(a^i b^j)$ can have different structures for a given $G$ depending on the choices of $i$ and $j$.  This is because, in general, $1\neq (a^i b^j)^q\in Z(G)$.  In particular, it need not be isomorphic to any of $\BZ_q\oplus \BZ_c\oplus \BZ_l$, $\BZ_{clq}$, $\BZ_{cq}\oplus\BZ_l$, or $\BZ_c\oplus\BZ_{ql}$.  The characters, and thus structure, of $C_G(a^i b^j)$ is ostensibly important to determining the indicators of a number of irreducible $\D(G)$-modules (see Proposition \ref{MOrho}).   However, we will soon see that we only need to know $\chi|_{Z(G)}$, where $\chi$ is any irreducible character of $C_G(a^i b^j)$.  In particular, we only need to know the irreducible characters, and thus structure of, $Z(G)$, which we have already determined in Lemma \ref{kqcent}.


\section{Representations of $\G$}\label{kqrepsect}
We continue to let $\G$ be as in Definition \ref{presentation}.  In the previous section we determined the multiplication and conjugacy classes of $\G$, so we can now proceed to determine its representation theory, and subsequently compute the indicators of its irreducible modules.  Our reference for the character theory of finite groups, especially induced characters, will be \cite{I}.  For any group $G$, we will always denotes the group of irreducible characters of $G$ by $\hat{G}$.  Furthermore, we denote the trivial character of a group $G$ by $1_{\hat{G}}$, or by $1$ when the group in question is to be understood from the context.

\begin{lem}\label{kqnonlinchar}
  Let $G=\G$ be as in Definition \ref{presentation}.  Let $H=\cyc{a,b^q}\leq G$, an abelian normal subgroup.  Let $\phi_{r,s}=\xi_r\otimes\psi_s\in \widehat{\BZ_k}\otimes\widehat{\BZ_l}$ be the linear character of $H$ given by $\phi_{r,s}(a^i b^{qj})=\xi_r(a^i)\psi_s(b^{qj})=\mu_k^{ri} \mu_l^{sj}$, where $\mu_k$ and $\mu_l$ are any fixed primitive $k$-th and $l$-th roots of unity, respectively.  Then $\phi_{r,s}^G$, the induced character on $G$, is of dimension $q$.  It is irreducible $\iff$ $r\not\equiv 0\bmod k/c$.  Furthermore, there are $l(k-c)/q$ non-equivalent characters amongst these irreducible characters.
\end{lem}

\begin{proof}
  By the theory of induced characters (see \cite{I} chapter 5), we have
  $$\phi_{r,s}^G(1) = |G| \frac{\phi_{r,s}(1)}{|H|} = \frac{|G|}{|H|} = q,$$
  whence the induced characters all have dimension $q$.

  By Lemma \ref{kqclass}, if $q\mid j$ we have
  $$\phi_{r,s}^G(a^ib^j) = kl \sum_{u=0}^{q-1} \frac{\phi_{r,s}(a^{in^u}b^j)}{kl} = \sum_{u=0}^{q-1} \phi_{r,s}(a^{in^u}b^j)$$
  and
  $$\phi_{r,s}^G(a^i b^j) = 0, \ j\not\equiv 0\bmod q.$$
  Note that the first equality is true even when $a^i\in Z(G)$, as it reduces to the value $q\phi_{r,s}(a^ib^j)$ as expected.

  To determine when $\phi_j^G$ is irreducible, we need only determine when $\ip{\phi_j^G}{\phi_j^G}=1$.  To this end,
  \begin{eqnarray}\label{indirrcheck}
    \ip{\phi_{r,s}^G}{\phi_{r,s}^G} &=& \frac{1}{kql}\sum_{g\in G} \phi_{r,s}^G(g)\overline{\phi_{r,s}^G(g)}\nonumber\\
    &=& \frac{1}{kql}\sum_{u=0}^{q-1}\sum_{v=0}^{q-1}\left( \sum_{x\in\cyc{a}}\sum_{y\in\cyc{b^q}} \phi_{r,s}(x^{n^u}y)\overline{\phi_{r,s}(x^{n^v}y)}\right)\nonumber\\
    &=& \frac{1}{kql} \sum_{u,v=0}^{q-1}\left( \sum_{y\in\cyc{b^q}} \psi_s(y)\overline{\psi_s(y)} \left( \sum_{x\in\cyc{a}} \xi_r(x^{n^u})\overline{\xi_r(x^{n^v})}\right)\right)\nonumber\\
    &=& \frac{1}{ql} \sum_{y\in\cyc{b^q}} \psi_s(y)\overline{\psi_s(y)}\left( \sum_{u,v=0}^{q-1} \ip{\xi_{rn^u}}{\xi_{rn^v}}\right)\nonumber\\
    &=& \frac{1}{q}\ip{\psi_s}{\psi_s}\sum_{u,v=0}^{q-1} \ip{\xi_{rn^u}}{\xi_{rn^v}}\nonumber\\
    &=& \frac{1}{q}\sum_{u,v=0}^{q-1} \ip{\xi_{rn^u}}{\xi_{rn^v}}
  \end{eqnarray}
  Now if $r\equiv 0\bmod k/c$ then $r n^u\equiv r n^v\equiv r\bmod k$ for all $u,v$ and so $\ip{\xi_{r n^u}}{\xi_{r n^v}}=1$ for all $u,v$.  Therefore, when $r\equiv 0\bmod k/c$ we have $(\ref{indirrcheck})=q>1$, and so $\phi_{r,s}^G$ is not irreducible.  On the other hand, if $r\not\equiv 0\bmod k/c$ then $\ip{\xi_{r n^u}}{\xi_{r n^v}}=1$ $\iff$ $n^{u-v}r \equiv r\bmod k$ $\iff$ $n^{u-v}\equiv 1\bmod k$ $\iff$ $u\equiv v\bmod q$.  Therefore, when $r\not\equiv 0\bmod k/c$ we have $\ip{\phi_{r,s}^G}{\phi_{r,s}^G}=1$, and so $\phi_{r,s}^G$ is irreducible as claimed.

  We now need only determine the number of isomorphism classes amongst these characters irreducible characters.  Similarly to the check for irreducibility, we find for $r,r'\not\equiv 0\bmod k/c$ that
  \begin{eqnarray*}
    \ip{\phi_{r,s}^G}{\phi_{r',s'}^G} &=& \frac{1}{q} \ip{\psi_s}{\psi_{s'}}\sum_{u,v=0}^{q-1} \ip{\xi_{rn^u }}{\xi_{n^v r'}},
  \end{eqnarray*}
  and similarly find that $\ip{\xi_{n^u r}}{\phi_{n^v r'}}=1$ $\iff$ $r' \equiv n^{s-t} r\bmod k$.  We conclude $\ip{\phi_{r,s}^G}{\phi_{r',s'}^G}=1$ $\iff$ $s\equiv s'\bmod l$ and $r'\equiv n^{h} r\bmod k$ for some $h$, and that $\ip{\phi_{r,s}^G}{\phi_{r',s'}^G}=0$ otherwise.

  Consequently, there are $l(k-c)/q$ distinct isomorphism classes amongst the irreducible characters as claimed.
\end{proof}

\begin{cor}\label{kqchar}
  Let $G=\G$ be as in Definition \ref{presentation}.  Then there are $cql$ irreducible linear characters of $G$.  They arise as the tensor product of the linear characters of $\cyc{b}\cong\BZ_{ql}$ and $\cyc{a^{k/c}}\cong\BZ_c$.  All irreducible characters of $G$ are either linear, or one of the $q$-dimensional irreducible representations from Lemma \ref{kqnonlinchar}.
\end{cor}
\begin{proof}
  Since, $\cyc{a}\lhd G$, the characters of $\cyc{b}=G/\cyc{a}\cong\BZ_{ql}$ are linear characters of $G$.  Additionally, the linear characters of $Z(G)=\cyc{a^{k/c},b^q}$ are linear characters of $G$.  Since $\cyc{b}\cap Z(G) = \cyc{b^q}$, it follows that the tensor product of irreducible linear characters of $\BZ_{ql}$ and $\BZ_c$ are irreducible linear characters of $G$.

  Now by Lemma \ref{kqnonlinchar}, we have
  $$cql + q^2\cdot\frac{l(k-c)}{q} = kql = |G|,$$
  and so these give all possible irreducible representations of $G$, establishing all claims.
\end{proof}

Now that we know the representation theory of $\G$, we can compute its indicators.

\begin{thm}\label{kq1dim}
  Let $G=\G$ be as in Definition \ref{presentation}.  Let $V$ be a 1-dimensional $G$-module with character $\chi=\xi_s\otimes \psi_t\in \widehat{\BZ_{ql}}\otimes \widehat{\BZ_{c}}$, as in Corollary \ref{kqchar}, where $\xi_s(b)=\mu_{ql}^s$ and $\psi_t(a)=\mu_c^t$, with $\mu_{ql}$ and $\mu_c$ primitive $ql$-th and $c$-th roots of unity respectively.  Then
  $$\nu_m(V) = \left\{ \begin{array}{cll} 1&;& c\mid mt\wedge ql\mid ms\\
  0&;& c\nmid mt \vee ql\nmid ms \end{array} \right.$$
\end{thm}
\begin{proof}
Since linear characters are multiplicative, we have
  \begin{eqnarray*}
    \nu_m(V) &=& \frac{1}{kql}\sum_{i=0}^{k-1}\sum_{j=0}^{q-1} \chi((a^i b^j)^m)\\
    &=& \frac{1}{kql} \sum_{i=0}^{k-1}\sum_{j=0}^{q-1} \chi^m(a^i b^j)\\
    &=& \frac{1}{kql} \sum_{i=0}^{k-1}\sum_{j=0}^{q-1} \xi_{ms}(b^j)\psi_{mt}(a^i)\\
    &=& \ip{\xi_{ms}}{1}\ip{\psi_{mt}}{1}\\
    &=& \left\{ \begin{array}{cll} 1 &;& c\mid mt \wedge ql\mid ms\\
    0&;& c\nmid mt \vee ql\nmid ms \end{array}\right..
  \end{eqnarray*}
\end{proof}

To compute subsequent indicators, we will need to recall the value $d$ from equation (\ref{dvalue}).  By Lemma \ref{sdpident} part (iv) and its proof, this value will naturally arise whenever we consider the order of an element of the form $a^i b^j$ in $G$ for $j\not\equiv 0\bmod q$.  Recall in the remarks after Lemma \ref{minq} that while $d\equiv 0\bmod k/c$, neither $d\equiv 0\bmod k$ nor $d\equiv k/c\bmod k$ held in general.

We start with a straightforward lemma that we will need in many of our subsequent calculations.

\begin{lem}\label{lsplitlem}
  Let $G=\G$ be as in Definition \ref{presentation}.  Let $\phi_{r,s}\in\widehat{\cyc{a,b^q}}$ be as in Lemma \ref{kqnonlinchar}, and define $d$ as in equation (\ref{dvalue}).  Denote the restriction of $\phi_{r,s}$ to $\cyc{a,b^{q^2}}$ by $\widetilde{\phi}_{r,s}$.  Then for any $m\equiv 0\bmod q$ we have
  \begin{eqnarray*}
    \frac{1}{kql}\sum_{x\in\cyc{a}}\sum_{y\in\cyc{b}\setminus\cyc{b^q}} \phi_{r,s}^G(x^{\frac{m}{q}d}y^{m})
    &=& q\ip{\phi_{\frac{m}{q}dr,\frac{m}{q}s}}{1}-\ip{\widetilde{\phi}_{\frac{m}{q}dr,\frac{m}{q}s}}{1}\\
    &=& \left\{ \begin{array}{cll}
        q-1 &;& kq\mid mdr \wedge lq\mid ms\\
        0 &;& kq\nmid mdr \vee \left(lq\nmid ms \wedge \left(q\nmid l \vee l\nmid ms\right)\right)\\
        -1 &;& q\mid l \wedge kq\mid mdr \wedge l\mid ms \wedge lq\nmid ms
    \end{array}\right.
  \end{eqnarray*}
\end{lem}
\begin{proof}
  The sum we wish to evaluate is clearly equal to
  \begin{eqnarray}\label{lsplit1}
    \frac{1}{kql}\sum_{x\in\cyc{a}}\left( \sum_{y\in\cyc{b}}\phi_{r,s}^G(x^{\frac{m}{q}d}y^{q\frac{m}{q}}) - \sum_{y\in\cyc{b^q}}\phi_{r,s}^G(x^{\frac{m}{q}d}y^{q\frac{m}{q}}) \right).
  \end{eqnarray}
  Then, since $\cyc{b^{q^2}}=\cyc{b^q} \iff q\nmid l$, we have
  \begin{eqnarray*}
    (\ref{lsplit1}) &=& \frac{1}{kql}\sum_{x\in\cyc{a}}\left(q\sum_{y\in\cyc{b^q}} \phi_{r,s}^G(x^{\frac{m}{q}d}y^{\frac{m}{q}}) - \gcd(q,l)\sum_{y\in\cyc{b^{q^2}}} \phi_{r,s}^G( x^{\frac{m}{q}d}y^{\frac{m}{q}})\right)\\
    &=& \sum_{u=0}^{q-1}\left(\ip{\phi_{\frac{m}{q}drn^u,\frac{m}{q}s}}{1}- \frac{1}{q}\ip{\widetilde{\phi}_{\frac{m}{q}drn^u,\frac{m}{q}s}}{1}\right)\\
    &=& q\ip{\phi_{\frac{m}{q}dr,\frac{m}{q}s}}{1}-\ip{\widetilde{\phi}_{\frac{m}{q}dr,\frac{m}{q}s}}{1},
  \end{eqnarray*}
  which gives the desired result.
\end{proof}
We note that the first equality above can be made valid with more general powers on $x$ and $y$.  We use this more specific version since it is what will appear in many of our remaining calculations.

\begin{thm}\label{kqnonlinear}
  Let $G=\G$ be as in Definition \ref{presentation}.  Let $V_{r,s}$ be an irreducible $q$-dimensional $G$-module with character $\phi_{r,s}^G$ as in Lemma \ref{kqnonlinchar}.
  \begin{enumerate}
  \item If $q\nmid l$, then
  $$\nu_m(V_{r,s}) = \left\{
  \begin{array}{cll}
    q   &;& q\mid m \wedge k\mid mr \wedge ql\mid ms\\
    q-1 &;& q\mid m \wedge k\nmid mr \wedge kq\mid mdr \wedge ql\mid ms\\
    1   &;& q\nmid m \wedge k\mid mr \wedge l\mid ms\\
    0   &;& k\nmid mr \vee \left(\left(q\nmid m\wedge l\nmid ms\right) \vee \left( q\mid m \wedge lq\nmid ms \right)\right)
  \end{array} \right..$$
  \item If $q\mid l$, then
  $$\nu_m(V_{r,s}) = \left\{
  \begin{array}{cll}
    q   &;& q\mid m \wedge k\mid mr \wedge lq\mid ms\\
    q-1 &;& q\mid m \wedge k\nmid mr \wedge kq\mid mdr \wedge lq\mid ms\\
    1   &;& q\nmid m \wedge k\mid mr \wedge l\mid ms\\
    \hspace{-9pt}-1  &;& q\mid m \wedge k\nmid mr \wedge kq\mid mdr \wedge l\mid ms \wedge lq\nmid ms\\
    0   &;& \mbox{otherwise}
  \end{array} \right.$$
  \end{enumerate}
\end{thm}
\begin{proof}
  First we consider the case $q\nmid m$.  Then $(a^i b^j)^m\not\in\cyc{a,b^q}$ $\iff$ $j\not\equiv 0\bmod q$, meaning $\phi_{r,s}^G((a^i b^j)^m)= 0$.  Therefore, since $\phi_{r,s}$ is a product of multiplicative characters, we have
  \begin{eqnarray*}
    \nu_m(V_{r,s}) &=& \frac{1}{kql} \sum_{g\in G}\phi_{r,s}^G(g^m)\\
    &=& \frac{1}{kql}\sum_{j=0}^{q-1}\sum_{i=0}^{k-1}\phi_{r,s}^G((a^i b^{qj})^m)\\
    &=& \frac{1}{kql}\sum_{u=0}^{q-1}\sum_{j=0}^{l-1}\sum_{i=0}^{k-1}\phi_{r,s}((a^{n^u i}b^{qj})^m)\\
    &=& \frac{1}{q}\sum_{u=0}^{q-1}\ip{\phi_{mrn^u,ms}}{1}\\
    &=& \ip{\phi_{mr,ms}}{1}\\
    &=& \left\{ \begin{array}{cll} 1&;& k\mid mr \wedge l\mid ms\\
    0&;& k\nmid mr \vee l\nmid ms \end{array}\right..
  \end{eqnarray*}

  Now assume $q\mid m$.  Then if $j\not\equiv 0\bmod q$ we have
  $$(a^i b^j)^m = a^{sd\frac{m}{q}}b^{mt} = a^{sd\frac{m}{q}}b^{q\frac{m}{q}t}.$$
  On the other hand, if $j\equiv 0\bmod q$ we have
  $$(a^i b^j)^m = a^{mi}b^{mj}.$$
  Thus, using Lemma \ref{lsplitlem}, and the  definition for $\widetilde{\phi}_{r,s}$ given there, we have
  \begin{eqnarray}
    \nu_m(V_{r,s}) &=& \frac{1}{kql} \sum_{x\in\cyc{a}}\sum_{y\in\cyc{b}} \phi_{r,s}^G((xy)^m)\nonumber\\
    &=& \frac{1}{kql}\sum_{x\in\cyc{a}}\left( \sum_{y\in\cyc{b^q}} \phi_{r,s}^G(x^m y^m) + \sum_{y\in\cyc{b}\setminus\cyc{b^q}} \phi_{r,s}^G(x^{\frac{m}{q}d}y^{q\frac{m}{q}})\right)\nonumber\\
    &=& \frac{1}{q}\sum_{u=0}^{q-1}\ip{\phi_{mrn^u,ms}}{1} + q\ip{\phi_{\frac{m}{q}dr,\frac{m}{q}s}}{1}-\ip{\widetilde{\phi}_{\frac{m}{q}dr,\frac{m}{q}s}}{1}\nonumber\\
    &=& \ip{\phi_{mr,ms}}{1} + q\ip{\phi_{\frac{m}{q}dr,\frac{m}{q}s}}{1}-\ip{\widetilde{\phi}_{\frac{m}{q}dr,\frac{m}{q}s}}{1}
  \end{eqnarray}
  Combining cases, applying Corollary \ref{classcor2} to simplify expressions, and using a few standard logical manipulations gives the desired formulas.
\end{proof}

This completes the calculation of the indicators for non-abelian groups of the form $\G$.  In summary, we have the following two results:

\begin{thm}\label{kqindsummary}
  Let $G\cong\G$ be as in Definition \ref{presentation}, and let $V$ be any irreducible $G$-module.  Then the indicators of $V$ satisfy $\nu_m(V)\in \{-1,0,1,q-1,q\}$ for all $m\in\BN$.  If $q\nmid l$, we further have $\nu_m(V)\in\{0,1,q-1,q\}$.
\end{thm}

\begin{thm}\label{kqortsummary}
  Let $G\cong\G$ be as in Definition \ref{presentation}.  Then $G$ is totally orthogonal $\iff$ $G$ is isomorphic to a dihedral group.
\end{thm}
\begin{proof}
  By Theorem \ref{kq1dim}, for $G$ to be totally orthogonal we must have $|Z(G)|\leq 2$ and $ql\leq2$.  Since $q$ is a prime, we thus must have $l=1$ and $q=2$.  Factoring $k=2^s x$ for some $s\geq 0$ and $x$ odd, the conditions $|Z(G)|\leq 2$ and $q=2$ force one of two possibilities:
  \begin{eqnarray}
    n\equiv -1 \bmod x&&n\equiv -1\bmod 2^s\\
    n\equiv -1 \bmod x&&n\equiv 2^{s-1}-1\bmod 2^s; \ s\geq 3.
  \end{eqnarray}
  The condition $n\equiv -1\bmod 2^s$ gives a dihedral group, which has $d=0$.  Checking Theorems \ref{kqnonlinear} and \ref{kq1dim}, we see that the dihedral groups are totally orthogonal (this is well-known).  If $n\equiv 2^{s-1}-1\bmod 2^s$, then we have $d\equiv k/2\bmod k$.  Again checking Theorems \ref{kqnonlinear} and \ref{kq1dim}, we find that $G$ would be totally orthogonal if and only if $2\mid j$ for every $j\not\equiv 0\bmod k/2$.  Since $s\geq 3$ forces $8\mid k$, this is impossible.
\end{proof}


\section{Indicators for $\D(\G)$}\label{kqdoubsect}
We continue to let $G=\G$ be as in Definition \ref{presentation}.  We wish now to compute the indicators for the irreducible modules over $\D(G)$.  We use the notation of Proposition \ref{MOrho} to denote the irreducible modules over $\D(G)$.  By Proposition \ref{FxF}, the indicators for a module $M(g,\rho)$ with $g\in Z(G)$ are entirely determined by the indicators of the $G$-module given by $\rho$, which we have already computed in the previous section.  Thus we will subsequently focus on the indicators of modules corresponding to non-singleton conjugacy classes.

\subsection{The sets $G_m(x)$}
By Corollary \ref{cor34}, we will need to compute the sets $G_m(x)$ from Definition \ref{Gmsetdef}.

\begin{prop}\label{kqsets} Let $G=\G$ be a group as in Definition \ref{presentation}. Let $i,s,u,v\in\BZ$.  Let $d$ be defined as in equation (\ref{dvalue}).  Then the following hold:
\begin{enumerate}
  \item If $q\mid v$ and $q\mid u$, then $$a^sb^v\in G_m(a^ib^u) \iff ql\mid mu \wedge k\mid mi$$
  \item If $q\nmid v$ and $q\mid u$, then $$a^s b^v\in G_m(a^ib^u) \iff ql\mid mu \wedge \left(( q\mid m \wedge kq \mid mdi ) \vee \left(q\nmid m \wedge i\sum_{h=0}^{m-1}n^{-hv}\equiv 0\bmod k\right)\right)$$
  \item If $q\mid v$ and $q\nmid u$, then $$a^sb^v\in G_m(a^i b^u) \iff ql\mid mu \wedge \left( \frac{m}{q}di + s\left(m - \frac{m}{q}d\right)\right) \equiv 0\bmod k$$
  \item If $u\not\equiv v\bmod q$, $q\nmid u$, and $q\nmid v$, then $$a^s b^v \in G_m(a^i b^u) \iff ql\mid mu \wedge kq \mid mdi$$
  \item If $q\nmid u$ and $v\equiv u\bmod q$, then $$a^s b^v \in G_m(a^i b^u) \iff ql\mid mu \wedge \left( mi + s\left(\frac{m}{q}d-m\right)\right)\equiv 0\bmod k$$
\end{enumerate}
\end{prop}
\begin{proof}To determine when $a^s b^t\in G_m(a^i b^j)$ for any $i,j,s,t\in\BN$, we wish to determine when the following identity holds:
  \begin{eqnarray}\label{gmelement}
    \prod_{h=0}^{m-1} (a^s b^t)^{-h} a^i b^j (a^s b^t)^h&=& \prod_{h=0}^{m-1} b^{-ht} a^{-s \sum_{p=0}^{h-1} n^{pt}} a^i b^j a^{s\sum_{p=0}^{h-1} n^{pt}} b^{ht}\nonumber\\
    &=& \prod_{h=0}^{m-1} b^{-ht} a^i a^{s(n^j -1)\sum_{p=0}^{h-1}n^{pt}} b^{ht} b^j\nonumber\\
    &=& \prod_{h=0}^{m-1} a^{n^{-ht}(i + s(n^j-1)\sum_{p=0}^{h-1}n^{pt})} b^j\nonumber\\
    &=& a^{\sum_{h=0}^{m-1}\left(n^{h(j-t)}(i+ s (n^j-1)\sum_{p=0}^{h-1} n^{pt})\right)} b^{mj}\\
    &=& 1.\nonumber
  \end{eqnarray}
  Thus for the identity to hold we must have $mj\equiv 0\bmod ql$, and in particular that $q\mid m$ or $q\mid j$.  We assume for the rest of the proof that $mj\equiv 0\bmod ql$.

  If $q\mid j$, then we are in the cases of i) and ii), and (\ref{gmelement}) equals $a^{i\sum_{h=0}^{m-1} n^{-ht}}$.  So for this element to equal 1 we must have $i \sum_{h=0}^{m-1}n^{-ht}\equiv 0\bmod k$.  If $q\mid t$, this is equivalent to $mi\equiv 0\bmod k$.  Else if $q\mid m$ it is equivalent to $mdi/q\equiv 0\bmod k$.  This proves i) and ii).

  On the other hand, suppose $q\nmid j$.  Then for (\ref{gmelement}) to be the identity we must have $q\mid m$ and that
  $$\sum_{h=0}^{m-1}\left(n^{h(j-t)}\left(i+ s (n^j-1)\sum_{p=0}^{h-1} n^{pt}\right)\right) \equiv 0\bmod k.$$
  If $q\mid t$, this is equivalent to
    \begin{eqnarray*}
      \sum_{h=0}^{m-1}\left(n^{hj}\left(i+ sh (n^j-1)\right)\right) &\equiv& i\sum_{h=0}^{m-1}n^{hj} + s(n^j-1)\sum_{h=0}^{m-1} hn^{hj}\\
      &\equiv& \frac{mdi}{q} + s(n^j-1)\sum_{h=0}^{m-1} hn^{hj}\\
      &\equiv& 0\bmod k.
    \end{eqnarray*}
  To interpret the last summation modulo $k$, we first consider the functions
  \begin{eqnarray}
    f_m(x) = \frac{1}{\ln(n)} \sum_{h=0}^{m-1} n^{hx} = \frac{1}{\ln(n)}\frac{n^{mx}-1}{n^x-1}.
  \end{eqnarray}
  Taking the derivative of all sides with respect to $x$, we get
  $$f_m'(x) = \sum_{h=0}^{m-1} h n^{hx} = \frac{ m n^{mx}(n^x-1) - n^x(n^{mx}-1)}{(n^x-1)^2}.$$
  In order to get a sensible closed form for $(n^j-1)f_m'(x)$ modulo $k$, we rewrite the last identity to get
  \begin{eqnarray}
    \sum_{h=0}^{m-1} h n^{hx} = \frac{m n^{mx}}{n^x-1} - \frac{n^x}{n^x-1}\sum_{h=0}^{m-1}n^{hx}.
  \end{eqnarray}
  We thus conclude that
  \begin{eqnarray*}
    (n^j-1)\sum_{h=0}^{m-1} h n^{hj} &=& m n^{mj} - n^j \sum_{h=0}^{m-1}n^{hj}\\
    &=& m n^{mj} - \sum_{h=0}^{m-1}n^{(h+1)j}\\
    &=& m n^{mj} - \sum_{h=1}^{m}n^{hj}
  \end{eqnarray*}
  as integers,  thus as equivalence classes modulo $k$.  Since $q\mid m$ and by assumptions on $n$, $n^{qm}\equiv 1\bmod k$.  So when $j\not\equiv 0\bmod q$ and $q\mid v$, it follows that $a^s b^v\in G_m(a^i b^j)$ if and only if
  $$\frac{md}{q}i + s\left(m-\frac{m}{q}d\right)\equiv 0\bmod k,$$
  which proves iii).

  To complete the proof, we now need to consider the case $q\nmid j$ and $q\nmid t$.  As before, we conclude that $q\mid m$ if the element in (\ref{gmelement}) is equal to the identity.  By assumptions on $j,t$, there is an integer $0<u_j^t<q$ with $u_j^t t\equiv j\bmod q$.  Then, over the rational numbers we have
  \begin{eqnarray*}
    (n^j-1)\sum_{p=0}^{h-1}n^{pt} &=& \frac{n^j-1}{n^t-1}(n^{ht}-1)\\
    &=& \frac{n^{u_j^t t}-1}{n^t-1}(n^{ht}-1)\\
    &=& \left(\sum_{p=0}^{u_j^t-1}n^{pt}\right)(n^{ht}-1).
  \end{eqnarray*}
  This last equality clearly holds over the integers, and therefore as equivalence classes modulo $k$.  Subsequently, for (\ref{gmelement}) to be the identity, we must have
  $$i\sum_{h=0}^{m-1}n^{h(j-t)} + s \left(\sum_{p=0}^{u_j^t-1}n^{pt}\right)\left( \sum_{h=0}^{m-1}n^{ht} - \sum_{h=0}^{m-1}n^{h(j-t)}\right)\equiv 0\bmod k.$$
  For $j\equiv t\bmod q$, we have $u_j^t=1$ and subsequently that this condition is equivalent to
  $$mi + s(md/q-m)\equiv 0\bmod k.$$
  When $j\not\equiv t\bmod q$, the congruence is equivalent to $imd/q \equiv 0\bmod k$.  This proves iv) and v).

  This completes all possible cases, and establishes the desired claims.
\end{proof}

The last part of ii) is the best congruence possible, in the following sense.  Taking $k=10^5-1, q=5, n=10, v=4, u=0, m=3$, we get that
$$\sum_{h=0}^2 n^{-vh} = \sum_{h=0}^2 n^h = 1 + 10 + 10^2 = 111,$$
and $\gcd(111,10^5-1)=3$.  Thus there are solutions other than $i\equiv 0\bmod k$.

Combined with Corollary \ref{indicforms}, the following corollary establishes the claim following Lemma \ref{kqnormalize} that we do not need to know the full character theory of the subgroups $C_G(a^i b^j)$ for $j\not\equiv 0\bmod q$, but only those of $Z(G)$.

\begin{cor}\label{kqsetcent}
  Under the hypotheses of Proposition \ref{kqsets}, if $q\nmid u$ and $g\in G_m(a^i b^u)$, then $g^m\in Z(G)$.
\end{cor}
\begin{proof}
  We consider each of the cases iii), iv), and v) from Proposition \ref{kqsets}.  Suppose $g\in G_m(a^i b^u)$.  Then we must have $ql\mid mu$, and in particular $q\mid m$.  If $g=a^s b^v$ for any $v\not\equiv 0\bmod q$, then $g^q = a^{ds}b^{qv}$.  By Lemma \ref{minq} and \ref{kqcent}, $a^{ds}\in Z(G)$ and $b^{qv}\in Z(G)$, so we conclude that $g^q\in Z(G)$, and thus $g^m\in Z(G)$.  On the other hand, if $g=a^s b^v$ with $q\mid v$, then rearranging the second condition of iii) we find that $sm\equiv \frac{m}{q}d(s-i)$, and so again $g^m = a^{sm}b^{mv}\in Z(G)$.
\end{proof}

Now that we have determined the sets $G_m(x)$, we can proceed to calculate the indicators for the irreducible modules over $\D(G)$.  By Proposition \ref{MOrho} and Lemma \ref{kqnormalize}, this will split naturally into two cases depending on the conjugacy class we consider.  For ease of exposition, we give these special names.

\begin{df}\label{typedef}
  Let $G=\G$ and for $g\in G$ let $V=M(\class(g),\rho)$ be an irreducible $\D(G)$-module, as in Proposition \ref{MOrho}.  If $g$ satisfies part (ii) of Lemma \ref{kqclass}, we say that $V$ is a Type I module.  If $g$ satisfies part (iii) of Lemma \ref{kqclass}, we say that $V$ is a Type II module.
\end{df}

As mentioned before, those $g\in Z(G)$ yield $\D(G)$-modules whose indicators are governed entirely by Lemma \ref{FxF} and the results of Section \ref{kqrepsect}.

\subsection{The Type I Modules}
We now determine the indicators of the Type I modules, as defined in Definition \ref{typedef}.

\begin{thm}\label{kqdouba}
  Let $G\cong\G$ be a group as in Definition \ref{presentation}.  Define $d$ as in equation (\ref{dvalue}). Suppose $a^i\not\in Z(G)$ and let $j\in\BZ$ with $q\mid j$.  Let $\chi_{r,s}$ be any irreducible character of $\cyc{a,b^q}=C_G(a^ib^j)$, defined as in Lemma \ref{kqnonlinchar}.  Consider the irreducible $\D(G)$-module $V=M(\class(a^ib^j),\chi_{r,s})$ from Proposition \ref{MOrho}.

  \begin{enumerate}
    \item If $ql\nmid mj$, then $\nu_m(V)=0$.
    \item If $q\nmid l$ and $ql\mid mj$, then
    \begin{eqnarray*}
        \nu_m(V) &=& \left\{ \begin{array}{cll}
        q   &;& q\mid m \wedge k\mid mi \wedge k\mid mr \wedge lq\mid ms\\
        q-1 &;& q\mid m \wedge\left(k\nmid mi\vee k\nmid mr\right) \wedge kq\mid mdi \wedge kq\mid mdr \wedge lq\mid ms\\
        1   &;& \left(q\nmid m  \vee lq\nmid ms\right) \wedge k\mid mi \wedge k\mid mr \wedge l\mid ms\\
        0   &;& \mbox{otherwise} \end{array} \right.
    \end{eqnarray*}
    \item If $q\mid l$ and $ql\mid mj$, then
    \begin{eqnarray*}
        \nu_m(V) &=& \left\{ \begin{array}{cll}
        q   &;& q\mid m \wedge k\mid mi \wedge k\mid mr \wedge lq\mid ms\\
        q-1 &;& q\mid m \wedge \left( k\nmid mr \vee k\nmid mi\right)\wedge kq\mid mdi \wedge lq\mid ms\\
        1   &;& q\nmid m \wedge k\mid mi \wedge k\mid mr \wedge l\mid ms\\
        \hspace{-9pt}-1   &;& q\mid m \wedge \left( k\nmid mi \vee k\nmid mr \right) \wedge lq\nmid ms \wedge kq\mid mdi \wedge kq\mid mdr \wedge l\mid ms\\
        0  &;& \mbox{otherwise}\end{array} \right.
    \end{eqnarray*}
  \end{enumerate}
\end{thm}
\begin{proof}
  To begin, suppose that we had $q\nmid m$, $q\mid j$, and $q\nmid t$.  Suppose for some $s\in\BZ$ that $a^s b^t\in G_m(a^i b^j)$.  Then it is easily observed that $(a^s b^t)^m$ is in no conjugate of $C_G(a^i b^j)=\cyc{a,b^q}$.  So by Lemma \ref{0contrib}, the element $a^s b^t$ contributes zero to the $m$-th indicator of $V$.  So, without loss of generality, by Lemma \ref{kqsets} we may suppose that if $q\mid v$ and $q\nmid t$, then
  \begin{eqnarray}
    a^s b^v\in G_m(a^i b^j) \iff ql\mid mj \wedge k\mid mi\label{gmsetrec1}\\
    a^s b^t \in G_m(a^i b^j) \iff ql\mid mj \wedge q\mid m \wedge kq \mid mdi.\label{gmsetrec2}
  \end{eqnarray}
  In particular, if $mj\not\equiv 0\bmod ql$, then it follows that $\nu_m(V)=0$.  So in the remainder of the proof we assume that $mj\equiv 0\bmod ql$.

  First we suppose that $q\nmid m$.  If $k\nmid mi$, then we again conclude that $\nu_m(V)=0$ by (\ref{gmsetrec1}).  So suppose that $k\mid mi$.  Then
  \begin{eqnarray*}
    \nu_m(V) &=& \frac{1}{kql}\sum_{g\in\class(a^i b^j)}\sum_{y\in\cyc{b^q}}\sum_{x\in\cyc{a}} \Tr(p_g\bowtie (xy)^m)\\
    &=& \frac{1}{kql}\sum_{u=0}^{q-1}\sum_{y\in\cyc{b^q}}\sum_{x\in\cyc{a}} \chi_{r,s}(x^{n^u m}y^m)\\
    &=& \frac{1}{kql}\sum_{u=0}^{q-1}\sum_{y\in\cyc{b^q}}\sum_{x\in\cyc{a}} \chi_{mn^ur,ms}(xy)\\
    &=& \frac{1}{q}\sum_{u=0}^{q-1}\ip{\chi_{mn^u r,ms}}{1}\\
    &=& \left\{ \begin{array}{cll}
        1 &;& k\mid mr \wedge l\mid ms\\
        0 &;& k\nmid mr \vee l\nmid ms
    \end{array}\right.
  \end{eqnarray*}
  This completes the case $q\nmid m$.

  So now suppose for the rest of the proof that $q\mid m$.  For notational convenience, set $\CO=\class(a^i b^j)$.  We consider the two sums
  \begin{eqnarray}
    \frac{1}{kql}\sum_{g\in\CO}\sum_{x\in\cyc{a}}\sum_{y\in\cyc{b^q}} \Tr(p_g\bowtie (xy)^m)\label{doubasum1}\\
    \frac{1}{kql}\sum_{g\in\CO}\sum_{x\in\cyc{a}}\sum_{y\in\cyc{b}\setminus\cyc{b^q}} \Tr(p_g\bowtie x^{\frac{m}{q}d}y^{q\frac{m}{q}})\label{doubasum2}.
  \end{eqnarray}
  We have $\nu_m(V)=(\ref{doubasum1})+(\ref{doubasum2})$, and we proceed to compute each part separately.

  By (\ref{gmsetrec1}) and Corollary \ref{0contrib}, we have
  \begin{eqnarray*}
    (\ref{doubasum1}) &=& \left\{ \begin{array}{cll}
        \frac{1}{q}\sum_{u=0}^{q-1}\ip{\chi_{mn^ur,ms}}{1} &;& k\mid mi\\
        0 &;& k\nmid mi
        \end{array}\right.\\
    &=& \left\{ \begin{array}{cll}
        \ip{\chi_{mr,ms}}{1} &;& k\mid mi\\
        0 &;& k\nmid mi
        \end{array}\right.\\
    &=& \left\{ \begin{array}{cll}
        1 &;& k\mid mi \wedge k\mid mr \wedge l\mid ms\\
        0 &;& k\nmid mi \vee k\nmid mr \vee l\nmid ms
        \end{array}\right.
  \end{eqnarray*}
  Additionally, by (\ref{gmsetrec2}), Lemma \ref{lsplitlem}, and Corollary \ref{0contrib}, we have
  \begin{eqnarray*}
    (\ref{doubasum2}) &=& \frac{1}{kql} \left\{ \begin{array}{cll}
        \displaystyle{\sum_{u=0}^{q-1}\sum_{x\in\cyc{a}}\sum_{y\in\cyc{b}\setminus\cyc{b^q}} \chi_{r,s}(x^{\frac{m}{q}dn^u}y^{m})} &;& kq\mid mdi\\
        0 &;& kq\nmid mdi
    \end{array}\right.\\
    &=&  \frac{1}{kql} \left\{ \begin{array}{cll}
        \displaystyle{\sum_{x\in\cyc{a}}\sum_{y\in\cyc{b}\setminus\cyc{b^q}} \chi_{r,s}^G(x^{\frac{m}{q}d}y^{m})} &;& kq\mid mdi\\
        0 &;& kq\nmid mdi
    \end{array}\right.\\
    &=& \left\{ \begin{array}{cll}
        \displaystyle{q\ip{\chi_{\frac{m}{q}dr,\frac{m}{q}s}}{1}-\ip{\widetilde{\chi}_{\frac{m}{q}dr,\frac{m}{q}s}}{1}} &;& kq\mid mdi\\
        0 &;& kq\nmid mdi
    \end{array}\right.
  \end{eqnarray*}
  In the last equality, $\widetilde{\chi}_{r,s}$ is defined as in Lemma \ref{lsplitlem}.

  Combining cases, using Corollary \ref{classcor2}, and simplifying, we get the desired formulas.
\end{proof}

\subsection{The Type II modules}

To complete our task, we must compute and analyze the indicators of the Type II modules, as defined in Definition \ref{typedef}.  We will need several technical lemmas to achieve this.  These will be essential in establishing that the indicators will all be integers (Corollary \ref{doubintind}), and when exactly they can be negative (Corollary \ref{negind}).

We start by defining a few things that will appear throughout our calculations in this section.
\begin{df}\label{doubbdef}
Let $G=\G$ be a group as in Definition \ref{presentation}.  Define $c,d$ as in equations (\ref{cvalue}) and (\ref{dvalue}) respectively.  Let $i,m,r\in\BZ$, and suppose that $\alpha$ is any generator for $\widehat{\cyc{a^{k/c}}}$.  We then make the following definitions:
\begin{enumerate}
    \item If $q\mid m$, define
    \begin{eqnarray}\label{hGdef}
      h_G(m) = \gcd\left(\frac{m}{q}(d-q),k\right),
    \end{eqnarray}
    and let $u\in\BZ$ (depending on $G$ and $m$) be such that
    \begin{eqnarray}\label{udef}
        u \frac{m}{q}(d-q)\equiv h_G(m)\bmod k.
    \end{eqnarray}
    \item If $q\mid m$ and $h_G(m)\mid mi$, we define
    \begin{eqnarray}\label{videf}
        v_i = \frac{mi}{h_G(m)}u\in\BZ
    \end{eqnarray}
    with $u$ as above, and then set
    \begin{eqnarray}\label{xidef}
        \xi_{m,i}^r = \alpha\left(a^{rd\frac{m}{q}v_i}\right).
    \end{eqnarray}
\end{enumerate}
\end{df}
Note that, by the definition of $h_G(m)$, the value $u$ in part (i) exists and can taken to be a unit modulo $k$.  We shall do so in the subsequent without further comment.

We now establish that $\xi_{m,i}^r$ is actually an integer under certain mild restrictions.  The first part of this Lemma will be useful in simplifying a few expressions later on.

\begin{lem}\label{doubblem}
Let $G=\G$ be a group as in Definition \ref{presentation}.  Define $c,d$ as in equations (\ref{cvalue}) and (\ref{dvalue}) respectively.  Let $i,m,r\in\BN$, and suppose that $\alpha$ is any generator for $\widehat{\cyc{a^{k/c}}}$.

Then the following two statements hold:
\begin{enumerate}
  \item If $q\mid m$, then
  $$q h_G(m) \mid mdi \iff h_G(m) \mid mi$$
  \item If $h_G(m)\mid mr$, then $\xi_{m,i}^r\in\{-1,1\}$ whenever it is defined.  Moreover, if $\xi_{m,i}^r = -1$, then $q=2$ and $2\mid k$.
\end{enumerate}
\end{lem}
\begin{proof}
  Let the value $u$ be defined as in Definition \ref{doubbdef}.  For the rest of the proof, we also let $u^{-1}\in\BZ$ denote any inverse for $u$ modulo $k$.  A direct check shows that
  \begin{eqnarray*}
    u^{-1}\equiv \frac{\frac{m}{q}(d-q)}{h_G(m)}\bmod \frac{k}{h_G(m)}.
  \end{eqnarray*}
  Thus, we may write
  \begin{eqnarray}\label{uinvdef}
    u^{-1} \equiv \frac{\frac{m}{q}(d-q)}{h_G(m)} + t\frac{k}{h_G(m)}\bmod k
  \end{eqnarray}
  for some $t\in\BZ$.  We note that in general we cannot always takes $t=0$.  The issue is akin to the congruence $7\cdot 7^2\equiv 8\cdot 7^2\bmod 7^2 11$, where $7$ is itself not a unit modulo $7^2 11$ but it acts as the unit $8$ on $7^2$.

  Now for i) we have that $q h_G(m) \mid m(d-q)$ by definition.  The equivalence then follows immediately by multiplying on the right by $i$.

  We now need to prove part ii), which we will do by cases.  Suppose for the rest of the proof that $h_G(m)\mid mr$--which is equivalent to $q h_G(m)\mid mrd$ by part i)--and that $\xi_{m,i}^r$ is defined.  Note that, since $Z(G)$ is a cyclic group of order $c$ by Lemma \ref{kqcent}, $\xi_{m,i}^r$ is a $c$-th root of unity.

  In all cases, we will consider the quantity
  \begin{eqnarray*}
    r d \frac{m}{q} v_i u^{-2} &\equiv& rd \frac{m}{q} \frac{mi}{h_G(m)} \frac{\frac{m}{q}(d-q)}{h_G(m)} + rd t \frac{m}{q}\frac{mi}{h_G(m)} \frac{k}{h_G(m)}\bmod k\\
    &=& rd \frac{m}{q} \frac{mi}{h_G(m)} \frac{\frac{m}{q}(d-q)}{h_G(m)} + t\frac{mr}{h_G(m)} \frac{mi}{h_G(m)} \frac{d k}{q}\bmod k.
  \end{eqnarray*}
  If $q\nmid k$, the last term is clearly zero.  And if $q\mid k$, then by Corollary \ref{classcor} we have $q\mid d$, and once again the last term is zero.  In conclusion
  \begin{eqnarray}\label{doublemid1}
    r d \frac{m}{q} v_i u^{-2} \equiv rd \frac{m}{q} \frac{mi}{h_G(m)} \frac{\frac{m}{q}(d-q)}{h_G(m)} \bmod k.
  \end{eqnarray}
  It is this expression which we aim to show is either zero or annihilated by $2$.  If so, then since $u^{-1}$ is a unit we will have
  \begin{eqnarray*}
    2 r d \frac{m}{q} v_i \equiv 0\bmod k,
  \end{eqnarray*}
  and subsequently that $\xi_{m,i}^r=\pm 1$.

  Suppose first that $q=2$.  Then, using that $d(d-q)\equiv 0\bmod k$ by Lemma \ref{minq}, we find
  \begin{eqnarray*}
    2 rd\frac{m}{q}v_i u^{-2} \equiv \frac{m}{q}\frac{mr}{h_G(m)} \frac{mi}{h_G(m)} d(d-q) \equiv 0\bmod k,
  \end{eqnarray*}
  where the fractions are integers by assumptions and part i).  We conclude that (\ref{doublemid1}) is annihilated by $2$. Subsequently $\xi_{m,i}^r$ is a complex number of order $2$, whence $\xi_{m,i}^r\in\{-1,1\}$ as desired.  Notice that if $2\nmid k$ then $2$ is a unit modulo $k$, and subsequently that $\xi_{m,i}^r =1$ in this case.  We will prove shortly (in Lemma \ref{negxi}) that $\xi_{m,i}^r=-1$ is actually possible in this case, and provide necessary and sufficient conditions for such $m,i,r$ to exist.

  Next we suppose that $q>2$ and that $q\nmid k$.  Then
  $$rd\frac{m}{q}v_i u^{-2}\equiv \frac{m}{q}\frac{mr}{h_G(m)} \frac{mi}{h_G(m)} \frac{1}{q} d(d-q)\equiv 0\bmod k,$$
  where the first three fractions are integers by assumptions and part i), and by $1/q$ is meant the inverse of $q$ modulo $k$.  We conclude that (\ref{doublemid1}) is zero in this case, and thus $\xi_{m,i}^r=1$ as well.

  Now if $q>2$ and $q\mid k$, then by Corollary \ref{classcor} we conclude that $q\mid d$.  So by Lemma \ref{classcor3}.ii, we get that $kq\mid d(d-q)$.  Subsequently, we conclude that $(\ref{doublemid1}\equiv 0\bmod k$, and therefore that $\xi_{m,i}^r=1$.

  The remaining claim in ii) is now immediate.
\end{proof}
The question then naturally arises as to when, exactly, $\xi_{m,i}^r=-1$ can occur in the above Lemma.  We answer this question with our next result.

\begin{lem}\label{negxi}
  Let $G=\G$ be a group as in Definition \ref{presentation}.  Factor $k=2^s x$, with $s\geq 0$, $2\nmid x$.  Define the quantities $\xi_{m,i}^r$ and $h_G(m)$ as in Definition \ref{doubbdef}, and define $c$ as in equation (\ref{cvalue}).  Then the following are equivalent:
  \begin{enumerate}
    \item $\exists m,i,r\in\BN$ such that $h_G(m)\mid mr$, $\xi_{m,i}^r$ is defined, and $\xi_{m,i}^r=-1$.
    \item $q=2$ and $\exists i,r\in\BN$ such that $h_G(2)\mid 2r$, $\xi_{2,i}^r$ is defined, and $\xi_{2,i}^r=-1$.
    \item $q=2$, $2\mid c$ and for $i=r=c/2$, then $\xi_{2,i}^r$ is defined and $\xi_{2,i}^r=-1$.
    \item $q=2$, $s\geq 3$, and $n\equiv 2^{s-1}\pm 1\bmod 2^s$.
  \end{enumerate}
  In general, however, $\xi_{m,i}^r=-1$ does not necessarily imply that $\xi_{2,i}^r$ is defined and satisfies $\xi_{2,i}^r=-1$.
\end{lem}
\begin{proof}
  By Lemma \ref{doubblem} we have that i) and ii) imply that $q=2$ and $s\geq 1$.  Conversely, if $q\neq 2$ or $s=0$, then $\xi_{m,i}^r=1$ for all $m,i,r\in\BN$ for which it is defined.  So for the remainder of the proof, we assume that $q=2$ and $s\geq 1$.  We also fix the same notation and definitions as in Lemma \ref{doubblem} and its proof.

  Let $m,i,r$ be such that $\xi_{m,i}^r$ from Lemma \ref{doubblem} is defined.  Then by assumptions the value $u$ in Lemma \ref{doubblem} is an odd number.  Let $d\equiv d'\frac{k}{c}\bmod k$.  Now by (the proof of) Lemma \ref{doubblem}, $\xi_{m,i}^r=-1$ is equivalent to
  \begin{eqnarray*}
    \frac{m}{2}r d v_i u^{-2} &\equiv& \frac{m}{2} \frac{mi}{h_G(m)} \frac{mr}{h_G(m)}\frac{d(d-q)}{2}\\
    &\equiv& \frac{m}{2} \frac{mi}{h_G(m)} \frac{mr}{h_G(m)} \frac{n-1}{c}d' \frac{k}{2}\\
    &\equiv& \frac{k}{2}\bmod k,
  \end{eqnarray*}
  which is equivalent to
  \begin{eqnarray}\label{negxi1}
    \frac{m}{2}\frac{mi}{h_G(m)}\frac{mr}{h_G(m)}d' \equiv 1\bmod 2.
  \end{eqnarray}
  If $d'$ is even, then this congruence can clearly never hold.

  Suppose for now that $d'\equiv 1\bmod 2$.  Then taking $m=2$, we have $h_G(2)=\gcd(d-2,k)=\gcd(n-1,k)=c$. Additionally, since $2\mid k$, we must have $2\mid c$ by Corollary \ref{classcor}.  So taking $i=r=c/2$, we have
  \begin{eqnarray*}
    \frac{2}{2}\frac{c}{c}\frac{c}{c}d' \equiv 1\bmod 2,
  \end{eqnarray*}
  as desired.  Thus $\xi_{2,\frac{c}{2}}^{\frac{c}{2}}=-1$.

  To complete the proof of the equivalences we need then only determine necessary and sufficient conditions for $d'\equiv 1\bmod 2$.  To this end, we consider the possible equivalency class of $n$ modulo $2^s$; this describes the contribution to $n$ from the $\Aut(\BZ_{2^s})$ summand of $\Aut(\BZ_k)$ (via the Chinese Remainder Theorem).  There are always at least two involutions in $\Aut(\BZ_{2^s})$, the identity and inversion.  If $s\geq 3$, then we have $\Aut(\BZ_{2^s})\cong \BZ_2\oplus \BZ_{2^{s-2}}$ (\cite[Theorem 7.3]{R}), giving us a total of four involutions from this summand of $\Aut(\BZ_k)$.  The possible congruences are $n\equiv \pm 1\bmod 2^s$ (the identity and inversion), and $n\equiv 2^{s-1}\pm 1\bmod 2^s$ if $s\geq 3$.  We proceed by cases on these congruences.

  If $n\equiv 1\bmod 2^s$, then $2^s\mid \gcd(n-1,k)=c$ and thus $\frac{k}{c}$ is odd, and $d$ is necessarily even.  Thus $d'\equiv 0\bmod 2$, and $\xi_{m,i}^r=-1$ can never occur.

  If $n\equiv -1\bmod 2^s$, then $c\in 4\BN-2$ and $2^s\mid d=n+1$.  Subsequently, $\frac{k}{c}=2^{s-1}x'$ with $2\nmid x'$.  Therefore $d'\equiv 0\bmod 2$, and again $\xi_{m,i}^r=-1$ can never occur.

  Now suppose $s\geq 3$ and $n\equiv 2^{s-1}+1\bmod 2^s$.  Then we have $2^{s-1}\mid c$, $2^s\nmid c$, and consequently $\frac{k}{c}\in 4\BN-2$.  Additionally, $d=n+1\in 4\BN-2$.  Thus $d'\equiv 1\bmod 2$, so by the above there exists $m,r,i$ with $\xi_{m,i}^r=-1$.

  Finally, suppose $s\geq 3$ and $n\equiv 2^{s-1}-1\bmod 2^s$.  Then $n-1\in 4\BN-2$, and so $c\in 4\BN-2$ and thus $2^{s-1}\mid \frac{k}{c}$ and $2^s\nmid \frac{k}{c}$.  Additionally, $n+1\equiv 2^{s-1}\bmod 2^s$, and therefore $d'\equiv 1\bmod 2$.  So by the above there exists $m,r,i$ with $\xi_{m,i}^r=-1$.

  This completes all possible cases and establishes the desired equivalences.

  For the last claim, suppose that $n\equiv 2^{s-1}-1\bmod 2^s$.  Then, as shown above, $2\mid c$ and $c/2$ is an odd number.  Suppose $p\mid k$, with $p>2$ a prime.  Let $s\in\BN$ be the largest natural number such that $p^s\mid k$.  Since $q=2$, we have $n\equiv \pm 1 \bmod p^s$.  If $n\equiv 1\bmod p^s$, then clearly $p^s\mid c$.  Else, if $n\equiv -1\bmod p^s$ then $p\nmid c$.  In particular, we conclude that $\gcd(c/2,k/c)=1$.  It then follows that $h_G(c)=\gcd(\frac{c}{2}(n-1),k) = c\cdot \gcd(\frac{c}{2},k/c) = c$.  It thus follows that $\xi_{c,1}^1=-1$ (and, in particular, is defined).  However, if $c\neq 2$, then $\xi_{2,1}^1$ is not defined, and we have already shown that we can arrange $c>2$ in this situation.
\end{proof}

We now need one more technical lemma before we can compute the indicators for the Type II modules.  This result is similar to Lemma \ref{lsplitlem}.

\begin{lem}\label{doubbsplit}
Let $G=\G$ be as in Definition 2.1.  Define $c$ and $d$ as in equations (\ref{cvalue}) and (\ref{dvalue}) respectively.  Let $j\in\BZ$ with $j\not\equiv 0\bmod q$ and $mj\equiv 0\bmod q$.  Let $\chi$ be any irreducible character of $C_G(a^i b^j)$, and let $\widehat{Z(G)}=\cyc{\psi,\phi}=\widehat{\BZ_{k/c}}\otimes \widehat{\BZ_{l}}$.  Let $r,s\in\BZ$ be such that $\chi|_{Z(G)}=\phi^r\otimes\psi^s$.  Let $\widetilde{\phi}$ denote the restriction of $\phi$ to $\cyc{b^{q^2}}$.  Finally, let $d'\in\BZ$ be such that $d\equiv d'\frac{k}{c}\bmod k$.  Then for any $m\in\BN$ with $q\mid m$, we have:

\renewcommand\arraystretch{2.25} 
  \begin{eqnarray*}
    \frac{1}{cql}\sum_{u\not\equiv 0,j\bmod q}\sum_{s=0}^{k-1} \chi(a^{\frac{m}{q}ds}b^{mu}) &=& \frac{k}{c}\ip{\psi^{\frac{m}{q}d'r}\phi^{\frac{m}{q}s}}{1} - \frac{2k}{cq}\ip{\psi^{r\frac{m}{q}d'}\widetilde{\phi}^{\frac{m}{q}s}}{1}\\
    &=& \left\{ \begin{array}{cll}
        \displaystyle{\frac{k(q-2)}{cq}} &;& k\mid mr \wedge lq\mid ms\\
        \displaystyle{-\frac{2k}{cq}}    &;& q\mid l \wedge k\mid mr \wedge l\mid ms\wedge lq\nmid ms\\
        0   &;& k\nmid mr \vee \left( lq\nmid ms \wedge \left( q\nmid l \vee l\nmid ms\right)\right)
        \end{array}\right.
  \end{eqnarray*}
\end{lem}
\begin{proof}
We consider the requirement $u\not\equiv 0,j\bmod q$.  When we sum over the possible values of $u$, we are thus excluding those $u$ with $u=qx$ or $u=qx+j$, some $x$.  Multiplying by $m$, we have $mu=mqx=q^2\frac{m}{q}x$ or $mu=q^2\frac{m}{q}u+mj\equiv q^2\frac{m}{q}u\bmod ql$.  Therefore, the sum we are interested in is equal to
    \begin{eqnarray*}
        \frac{1}{cql}\sum_{s=0}^{k-1}\left(\sum_{u=0}^{lq-1} \chi(a^{\frac{m}{q}ds}b^{mu}) - \sum_{u\equiv 0,j\bmod q}\chi(a^{\frac{m}{q}ds}b^{mu})\right).
    \end{eqnarray*}
This summation is then easily seen to be equal to
    \begin{eqnarray*}
        \frac{1}{cql} \sum_{x\in\cyc{a}}\left(q\sum_{y\in\cyc{b^q}}\psi^r(x^{\frac{m}{q}d}) \phi^s(y^{\frac{m}{q}}) - 2\gcd(q,l)\sum_{y\in\cyc{b^{q^2}}}\psi^r(x^{\frac{m}{q}d}) \phi^s(y^{\frac{m}{q}})\right).
    \end{eqnarray*}
Finally, we conclude that the summation is question is equal to
    \begin{eqnarray}
        \frac{k}{c}\ip{\psi^{\frac{m}{q}d'r}\phi^{\frac{m}{q}s}}{1} - \frac{2k}{cq}\ip{\psi^{r\frac{m}{q}d'}\widetilde{\phi}^{\frac{m}{q}s}}{1}.\label{asbucontrib}
    \end{eqnarray}
The value of this last expression depends on whether or not $q\mid l$; equivalently, whether or not $\cyc{b^{q^2}}\neq \cyc{b^q}$.  Specifically
  \renewcommand\arraystretch{2.25} 
  \begin{eqnarray*}
    (\ref{asbucontrib}) &=& \left\{ \begin{array}{cll}
        \displaystyle{\frac{k(q-2)}{cq}} &;& kq\mid mdr \wedge lq\mid ms\\
        \displaystyle{-\frac{2k}{cq}}    &;& q\mid l \wedge kq\mid mdr \wedge l\mid ms\wedge lq\nmid ms\\
        0   &;& kd\nmid mdr \vee \left( lq\nmid ms \wedge \left( q\nmid l \vee l\nmid ms\right)\right)
        \end{array}\right.
  \end{eqnarray*}
which gives the desired result.
\end{proof}

\begin{thm}\label{kqdoubb}
    Let $G=\G$ be as in Definition 2.1.  Define $c$ and $d$ as in equations (\ref{cvalue}) and (\ref{dvalue}) respectively.  Let $j\in\BZ$ with $j\not\equiv 0\bmod q$.  Let $\chi$ be any irreducible character of $C_G(a^i b^j)$, and let $\widehat{Z(G)}=\cyc{\psi,\phi}=\widehat{\BZ_{k/c}}\otimes \widehat{\BZ_{l}}$.  Let $r,s\in\BZ$ be such that $\chi|_{Z(G)}=\phi^r\otimes\psi^s$.  Define $V=M(\CO(a^i b^j),\chi)$ as in Proposition \ref{MOrho}. Finally, let $m\in\BN$.

    \begin{enumerate}
      \item If $ql\nmid mj$, then $\nu_m(V)=0$.
      \item If $q\nmid l$ and $ql\mid mj$, then
      \renewcommand\arraystretch{2.25} 
      \begin{eqnarray*}
      \nu_m(V)=\left\{ \begin{array}{cll}
        \displaystyle{\frac{1}{cq}\left(k(q-2) + 2 \xi_{m,i}^r h_G(m)\right)} &;& kq\mid mdi \wedge kq\mid mdr \wedge l\mid ms\\
        \displaystyle\frac{2 h_G(m)}{cq}\xi_{m,i}^r &;& \hspace{-5pt}
            \renewcommand\arraystretch{1}
            \begin{array}{l}
                h_G(m)\mid mi \wedge h_G(m)\mid mr \wedge l\mid ms\, \wedge\\
                \ \left( kq\nmid mdi \vee kq\nmid mdr\right)
            \end{array}\\
        0 &;& \mbox{otherwise} \end{array}\right.
      \end{eqnarray*}
      \item If $q\mid l$ and $ql\mid mj$, then
      \renewcommand\arraystretch{2.25} 
        \begin{eqnarray*}
            \nu_m(V)=\left\{ \begin{array}{cll}
                \displaystyle{\frac{1}{cq}\left(k(q-2) + 2 \xi_{m,i}^r h_G(m)\right)} &;& kq\mid mdi \wedge kq\mid mdr \wedge lq\mid ms\\
                \displaystyle{\frac{2 h_G(m)}{cq}\xi_{m,i}^r} &;& \hspace{-5pt}
                    \renewcommand\arraystretch{1}
                    \begin{array}{l}
                        h_G(m)\mid mi \wedge h_G(m)\mid mr \wedge l\mid ms\,\wedge\\
                        \ \left( kq\nmid mdi \vee kq\nmid mdr\right)
                    \end{array}\\
                \displaystyle{\frac{2}{cq}\left(\xi_{m,i}^r h_G(m)-k\right)} &;& kq\mid mdi \wedge kq\mid mdr \wedge l\mid ms \wedge lq\nmid ms\\
                0 &;& \mbox{otherwise} \end{array}\right.
        \end{eqnarray*}
    \end{enumerate}
\end{thm}
\begin{proof}
  We first note that $h_G(m)$ is defined wherever it occurs in the above formula.  Furthermore, by Lemma \ref{doubblem}.ii, $\xi_{m,i}^r=\pm 1$ whenever it occurs in the stated formula.

  By Proposition \ref{kqsets}, if $ql\nmid mj$, then $G_m(a^i b^j)=\emptyset$, whence $\nu_m(V)=0$.  So for the remainder of the proof we assume that $ql\mid mj$.  Since $q\nmid j$ by assumptions, we in particular assume that $q\mid m$.  We also let the values $v_i$ and $u$ be defined as in the statement of Lemma \ref{doubblem}.  Furthermore, Lemma \ref{minq}.iv implies $c\mid h_G(m)$ whenever $q\mid m$.  This fact will be necessary for reducing a number of summations that appear in the calculation of $\nu_m(V)$ to inner products of characters of $Z(G)$.  Finally, let $d'\in\BZ$ be such that $d\equiv d' \frac{k}{c}\bmod k$, which exists by Lemma \ref{minq}.

  Recall that by Corollary \ref{kqsetcent} that if $g\in G_m(a^i b^j)$ then $g^m\in Z(G)$.  So by Corollaries \ref{0contrib} and \ref{cor34}, we split the summation for $\nu_m(V)$ into the three pieces determined by cases (iii), (iv), and (v) of Proposition \ref{kqsets}.  To this end, we recall that $C_G(a^i b^j)$ is abelian by Proposition \ref{kqnormalize}, and so $\chi$ and $\chi|_{Z(G)}$ are both multiplicative.

  First we consider the contribution of elements of the form $a^tb^v$, with $q\mid v$.  By Proposition \ref{kqsets} and our assumptions, $a^t b^v\in G_m(a^i b^j)$ $\iff$ $t\frac{m}{q}(d-q)\equiv \frac{m}{q}di \bmod k$.  Such a value $t$ exists $\iff$ $q h_G(m)\mid mdi$, which is equivalent to $h_G(m)\mid mi$ by Lemma \ref{doubblem}.i.  So if $h_G(m)\nmid mi$, the elements of the form $a^t b^v$ contribute 0 to the summation.  Else, suppose $h_G(m)\mid mi$.  Then if $a^t b^v\in G_m(a^i b^j)$, we have
  \begin{eqnarray*}
    t = \frac{ mdi/q}{h_G(m)}u + \frac{k}{h_G(m)}t', \ 0\leq t'<h_G(m).
  \end{eqnarray*}
  Using this and the initial congruence, we find
  \begin{eqnarray*}
    \chi((a^t b^v)^m) &=& \chi(a^{mt}b^{mv})\\
    &=&\chi(a^{\frac{m}{q}d(t-i)})\chi(b^{mv})\\
    &=& \chi(a^{-\frac{m}{q}di}) \chi(a^{\frac{m}{q}dt})\phi^s(b^{mv})\\
    &=& \psi^r(a^{\frac{m}{q}d\left( \frac{mdi/q}{h_G(m)}u-i\right)}) \psi^{\frac{rmkd'}{h_G(m)}}(a^{t'\frac{k}{c}})\phi^s(b^{mv}).
  \end{eqnarray*}
  To simplify the first term, we observe that
  \begin{eqnarray*}
    \frac{mdi/q}{h_G(m)}u - i &\equiv& \frac{umdi/q}{h_G(m)} - \frac{ih_G(m)}{h_G(m)}\\
    &\equiv& \frac{iu(\frac{m}{q}d-\frac{m}{q}(d-q))}{h_G(m)}\\
    &\equiv& \frac{mi}{h_G(m)}u\\
    &\equiv& v_i\bmod k.
  \end{eqnarray*}

  Using Lemma \ref{doubblem}.i, we therefore find that, under our assumptions, these terms contribute the following to $\nu_m(V)$:
  \begin{eqnarray}\label{ascontrib}
    \frac{\xi_{m,i}^r}{cq} h_G(m) \ip{\psi^{\frac{mk}{qh_G(m)} r d'}}{1}\ip{\phi^{ms}}{1}
    &=& \left\{ \begin{array}{cll} \frac{\xi_{m,i}^r h_G(m)}{cq} &;& qh_G(m) \mid mrd \wedge l\mid ms\\
    0&;& \mbox{otherwise} \end{array}\right.\nonumber\\
    &=& \left\{ \begin{array}{cll} \frac{\xi_{m,i}^r h_G(m)}{cq} &;& h_G(m) \mid mr \wedge l\mid ms\\
    0&;& \mbox{otherwise} \end{array}\right.
  \end{eqnarray}

  Let us now consider the contribution of elements of the form $a^s b^u$, with $u\not\equiv 0,j\bmod q$.  Under our assumptions, we have that $a^s b^u\in G_m(a^i b^j) \iff kq\mid mdi$.  So if $kq\nmid mdi$, these terms contribute $0$ to $\nu_m(V)$.  So suppose, instead, that $kq\mid mdi$. We have that
  \begin{eqnarray*}
    \chi((a^s b^u)^m) = \chi(a^{\frac{m}{q}ds}b^{mu})=\psi^r(a^{\frac{m}{q}ds})\phi^s(b^{mu}).
  \end{eqnarray*}
  Since $u\not\equiv 0,j\bmod q$, when we sum these terms over $s$ and $u$, we are excluding those $u$ with $u=qx$ or $u=qx+j$, some $x$.  Multiplying by $m$, we have $mu=mqx=q^2\frac{m}{q}x$ or $mu=q^2\frac{m}{q}u+mj\equiv q^2\frac{m}{q}u\bmod ql$.  Thus, the contribution of these terms is given by
  \begin{eqnarray}
    \frac{1}{cql}\sum_{u\not\equiv 0,j\bmod q}\sum_{s=0}^{k-1} \chi((a^sb^u)^m) &=& \frac{1}{cql}\sum_{u\not\equiv 0,j\bmod q}\sum_{s=0}^{k-1} \chi(a^{\frac{m}{q}ds}b^{mu})
  \end{eqnarray}
  The value of this summation is given by Lemma \ref{doubbsplit}.

  Finally, we consider the contribution of the elements of the form $a^s b^v$ with $v\equiv j\bmod q$.  Under our assumptions, we have that $a^s b^v\in G_m(a^i b^j) \iff s\frac{m}{q}(q-d)\equiv mi\bmod k$.  Such an $s$ exists $\iff$ $h_G(m)\mid mi$.  So if $h_G(m)\nmid mi$, these terms contribute $0$ to $\nu_m(V)$.  Else, suppose $h_G(m)\mid mi$.  Then if $a^sb^v\in G_m(a^i b^j)$ we must have
  $$s = -v_i + \frac{k}{h_G(m)}s', \ 0\leq s'<h_G(m),$$
  and therefore
  \begin{eqnarray*}
    \chi((a^s b^v)^m) &=& \chi(a^{\frac{m}{q}ds}b^{mv})\\
    &=& \psi\left( a^{-\frac{m}{q} r d' v_i \frac{k}{c}}\right) \psi^r\left(a^{\frac{km}{q h_G(m)}d' s' \frac{k}{c}}\right)\phi^s(b^{mv})\\
    &=& \overline{\xi_{m,i}^r} \cdot \psi^{\frac{km}{q h_G(m)}r d'}\left(a^{s' \frac{k}{c}}\right)\phi^{\frac{m}{q}s}(b^{qv}).
  \end{eqnarray*}
  Now since $v\equiv j\bmod q$, we have that $v=qx+j$ for $0\leq x<l$.  Then $mv = q^2 \frac{m}{q}x+ mj \equiv q^2 \frac{m}{q}x\bmod ql$.  So, again, the contribution of these terms depends on whether or not $q\mid l$.  If $q\nmid l$, then the contribution is
  \begin{eqnarray*}
    \frac{1}{cql}\frac{h_G(m)}{c}\overline{\xi_{m,i}^r} \sum_{x\in\cyc{a^{k/c}}}\psi^{\frac{mk}{q h_G(m)}rd'}(x) \sum_{y\in\cyc{b^q}}\phi^{\frac{m}{q}s}(y) &=& \frac{h_G(m)}{cq}\overline{\xi_{m,i}^r} \ip{ \psi^{\frac{m}{q}rd'}\phi^{\frac{m}{q}s}}{1}.
  \end{eqnarray*}
  On the other hand, if $q\mid l$, then the contribution is
  \begin{eqnarray*}
    \frac{q}{cql}\frac{h_G(m)}{c}\overline{\xi_{m,i}^r}\sum_{x\in\cyc{a^{k/c}}} \psi^{\frac{mk}{q h_G(m)}rd'}(x) \sum_{y\in\cyc{b^{q^2}}}\phi^{\frac{m}{q}s}(y) &=& \frac{h_G(m)}{cq}\overline{\xi_{m,i}^r} \ip{ \psi^{\frac{m}{q}rd'}\phi^{\frac{m}{q}s}|_{\cyc{b^{q^2}}}}{1}.
  \end{eqnarray*}
  Combining, the contribution of these terms is
  \begin{eqnarray}\label{asbjcontrib}
    \left\{ \begin{array}{cll}
        \displaystyle{\frac{h_G(m)}{cq}\overline{\xi_{m,i}^r}} &;& h_G(m)\mid mr \wedge l\mid ms \wedge (lq\mid ms \vee q\mid l)\\
        0 &;& \mbox{ otherwise }
    \end{array}\right.
  \end{eqnarray}
  A simple argument then shows that
  \begin{eqnarray}
    q\mid m \Rightarrow \big( \left(l\mid ms \wedge (lq\mid ms\vee q\mid l)\right) \iff l\mid ms \big)
  \end{eqnarray}

  Combining all of the cases; equations (\ref{ascontrib}), (\ref{asbucontrib}), and (\ref{asbjcontrib}); and both parts of Lemma \ref{doubblem}, we get the desired formulas.
\end{proof}

This result completes the determination of all indicators over the double.  We remark that, since $C_G(a^i b^j)$ is an abelian group (containing $Z(G)$), any values of $r$ and $s$ (up to equivalence) can be obtained from some irreducible character of $C_G(a^i b^j)$.

The question now naturally arises as to what are necessary and sufficient conditions for a negative indicator to arise over $\D(G)$ when $q\nmid l$.  This is the precise reason we desired to prove Lemma \ref{negxi}.

\begin{cor}\label{negind}
  Let $G\cong\G$ be a group as in Definition \ref{presentation}.  Suppose that $q\nmid l$.  Factor $k=2^s x$, with $s\geq 0$, $2\nmid x$.  Consider the Hopf algebra $H=\D(G)$.  Then the following are equivalent:
  \begin{enumerate}
    \item $\exists m\in\BN$ and an irreducible (left) $H$-module $V$ with $\nu_m(V)<0$
    \item $\exists$ an irreducible (left) $H$-module $W$ with $\nu_2(W)=-1$.
    \item $q=2$, $s\geq 3$, and $n\equiv 2^{s-1}\pm 1\bmod 2^s$.
  \end{enumerate}
  In general, however, it need not be true that $\nu_m(W)<0\Rightarrow \nu_2(W)=-1$ for every irreducible $H$-module $W$.
\end{cor}
\begin{proof}
  By Theorem \ref{kqdoubb}, Lemma \ref{FxF} applied to Theorem \ref{kqindsummary}, and Theorem \ref{kqdouba}, we conclude that if an irreducible $H$-module $W$ has $\nu_m(W)<0$ then $W$ must be a Type II module.  Its indicators, in particular, are given by Theorem \ref{kqdoubb}.  Thus, $\nu_m(W)<0$ forces $\xi_{m,i}^r=-1$, for the appropriate $i,r$.

  So by Lemma \ref{negxi}, we must have that $q=2$, $s\geq 3$, and $n\equiv 2^{s-1}\pm 1\bmod 2^s$.  Furthermore, by Lemma \ref{negxi} and Theorem \ref{kqdoubb}, whenever these conditions on $q$ and $n$ hold then there exist irreducible $H$-modules $W_1,W_2$ with $\nu_m(W_1)<0$ and $\nu_2(W_2)=-1$.

  This establishes the desired equivalences.

  For the last claim, we use the corresponding result proved in Lemma \ref{negxi}.  In particular, taking $k=24$, $n=19$, $m=c=6$, $i=1$, $r=1$, and any $H$-module $W$ yielding the values $i$ and $r$, then $\nu_6(V)=-1$ and $\nu_2(V)=0$.
\end{proof}

Summarizing, we have
\begin{thm}\label{doubintind}
    Let $G\cong\G$ be a group as in Definition \ref{presentation}.  Let $V$ be any (left) module over $\D(G)$.  Then $\nu_m(V)\in\BZ$ for all $m\in\BN$.  Moreover, if $q\nmid l$ and either $q\neq 2$ or $8\nmid k$, then $\nu_m(V)\in\BN\cup\{0\}$ for all $m\in\BN$.
\end{thm}
\begin{proof}
  By Lemma \ref{FxF} applied to Corollary \ref{kqindsummary}, Theorem \ref{kqdouba}, and Theorem \ref{kqdoubb}, any indicator of $V$ is clearly a rational number.  Since indicators are all algebraic integers in a cyclotomic field (\cite[pp. 17]{KSZ2}), the indicators are thus all necessarily integers.  The final claim is a weaker version of Corollary \ref{negind}.
\end{proof}

Since the special case $d\equiv 0\bmod k$ and $l=1$ arises in the case of the dihedral groups (see Section \ref{dnexample}) and groups of order $pq$ (see Section \ref{pqexample}), we restate Theorems \ref{kqdoubb} and \ref{doubintind} for this case.
\begin{cor}\label{kqdoubb0}
  Let $G\cong\BZ_k\rtimes_n\BZ_q$ be a group as in Definition \ref{presentation}.  Define $d$ as in equation (\ref{dvalue}), and suppose $d\equiv 0\bmod k$.  Let $j\in\BN$ with $j\not\equiv 0\bmod q$.  Let $\chi$ be an irreducible character of $C_G(a^i b^j)$.  Then the irreducible $\D(G)$-module $V=M(\class(a^i b^j),\chi)$ from Proposition \ref{MOrho} has the following indicators:
  $$\nu_m(V) = \left\{ \begin{array}{cll} \displaystyle{\frac{1}{cq}\left( 2\gcd(m,k) + k(q-2) \right)}&;&q\mid m\\
  0&;&q\nmid m  \end{array} \right.$$
\end{cor}

\begin{cor}\label{d0intind}
    Let $G\cong\BZ_k\rtimes_n\BZ_q$ be a group as in Definition \ref{presentation}.  Define $d$ as in equation (\ref{dvalue}), and suppose $d\equiv 0\bmod k$.  Then for any $\D(G)$-module $V$, we have $\nu_m(V) \in \BN\cup\{0\}$ for all $m\in\BN$.
\end{cor}

\section{Examples}\label{examplesect}
In this section we apply our results to compute the indicators for a number of interesting groups and their doubles.
\subsection{The Dihedral Groups}\label{dnexample}
Let $D_k=\cyc{a,b \ | \ a^k = b^2 =1, bab = a^{-1}}$ be the dihedral group of order $2k$.  We fix this representation throughout his section.  Then $D_k \cong \BZ_k \rtimes_{-1}\BZ_2$ in the notation of Definition \ref{presentation}, so we can use our results to find the indicators for $D_k$ and its double $\D(D_k)$.

We first state the basic facts about these groups needed to state these indicators, the proofs of which are trivial and independently well-known:
\begin{cor}\label{dkbasics}
  Consider $D_k\cong \BZ_k\rtimes_{-1}\BZ_2$, the dihedral group of order $2k$.  In the notation of Sections \ref{kqsect} and \ref{kqdoubsect},
  \begin{enumerate}
    \item Let $c=|Z(D_k)|$.  Then $c=1 \iff 2\nmid k$ and $c=2\iff 2\mid k$
    \item $d=0$
    \item If $k$ is odd, then $D_k$ has two irreducible $1$-dimensional representations and $\frac{n-1}{2}$ irreducible $2$-dimensional representations.  They are given by Lemma \ref{kqnonlinchar} and Corollary \ref{kqchar}, and these give all of the irreducible representations of $D_k$.
    \item If $k$ is even, then $D_k$ has four irreducible $1$-dimensional representations and $\frac{n}{2}-1$ irreducible $2$-dimensional representations.  They are given by Lemma \ref{kqnonlinchar} and Corollary \ref{kqchar}, and these give all of the irreducible representations of $D_k$.
  \end{enumerate}
\end{cor}

\begin{thm}
  Let $V$ be a 1-dimensional $D_k$-module with character $\chi=\xi_s\otimes\psi_t$ as in Theorem \ref{kq1dim}.  Then
  $$\nu_m(V) = \left\{ \begin{array}{cll} 1 &;& 2\mid ms\\
  0&;& 2\nmid ms \end{array}\right..$$
  In particular, the trivial character has an indicator of $1$ for every $m$.  The non-trivial linear characters have indicator 1 when $2\mid m$ and $0$ otherwise.
\end{thm}
\begin{proof}
  This follows immediately from Theorem \ref{kq1dim} and Corollary \ref{dkbasics}.
\end{proof}

\begin{thm}
  Let $V$ be an irreducible 2-dimensional $D_k$-module with character $\phi_j^{D_k}$, with notation as in Lemma \ref{kqnonlinchar}.  Then
  $$\nu_m(V) = \left\{ \begin{array}{cll} 2 &;& 2\mid m \wedge k\mid mj \\
  1 &;& 2\mid m \oplus k\mid mj\\
  0 &;& 2\nmid m \wedge k\nmid mj \end{array}\right.$$
\end{thm}
\begin{proof}
  This follows immediately from Corollary \ref{dkbasics} and Theorem \ref{kqnonlinear}.
\end{proof}

Now we can state the indicators for the irreducible $\D(D_k)$-modules.  Recall first that those modules in Proposition \ref{MOrho} coming from singleton classes can have their indicators computed directly from the previous two theorems and Proposition \ref{FxF}.  For the remaining modules, we have the following two theorems:

\begin{thm}
  Suppose $a^i\not\in Z(D_k)$ and let $V=M(\class(a_i),\phi_j)$ be the irreducible $\D(D_k)$-module given by Proposition \ref{MOrho}, where $\phi_j$ is the irreducible character of $\cyc{a}$ given by $\phi_j(a)=\mu_k^j$, with $\mu_k$ a fixed primitive $k$-th root of unity.  Then
  $$\nu_m(V) = \left\{ \begin{array}{cll} 2 &;& 2\mid m \wedge k\mid mj \wedge k\mid mi\\
  1&;& 2\mid m \oplus \left( k\mid mj\wedge k\mid mi \right)\\
  0&;& 2\nmid m \wedge k\nmid mj \wedge k\nmid mi. \end{array}\right..$$
\end{thm}
\begin{proof}
  This follows from Corollary \ref{dkbasics} and Theorem \ref{kqdouba}.
\end{proof}

\begin{thm}
  Let $c=|Z(D_k)|$.  Consider the irreducible $\D(D_k)$-module $V=M(\class(a^i b),\chi)$ given in Proposition \ref{MOrho}, where $\chi$ is any irreducible representation (and thus linear character) of $C_{D_k}(a^i b)$.  Then
  $$\nu_m(V) = \left\{ \begin{array}{cll} \frac{\gcd(m,k)}{c}&;& 2\mid m\\
  0&;& 2\nmid m \end{array}\right..$$
\end{thm}
\begin{proof}
  The formula is a direct consequence of Corollary \ref{dkbasics} and Corollary \ref{kqdoubb0}.
\end{proof}

This completes the determination of all indicators for irreducible modules over $D_k$ and $\D(D_k)$.

\subsection{Non-abelian groups of order $pq$}\label{pqexample}
For another example of groups our results apply to, let us consider primes $p,q$ with $p>q$ and $q\mid p-1$.

\begin{lem} \cite[Thm. 4.20]{R}
Let $p,q$ be primes with $p>q$.  Let $G$ be a group of order $pq$.  Then $G$ is non-abelian $\iff$ $q\mid p-1$ and we have the presentation
\begin{eqnarray}\label{pqpres}
    G= \cyc{a,b \ | \ a^p=b^q=1, bab^{-1}=a^n}
\end{eqnarray}
for some $n\in\BN$ with $n^q\equiv 1\bmod p$ but $n\not\equiv 1\bmod p$.
\end{lem}
In particular, in the notation of Definition \ref{presentation}, we have that for any such group $G$ that $G\cong \BZ_p \rtimes_n \BZ_q$ and we can apply our results to the indicators of $G$ and $\D(G)$.

Fix now a non-abelian group $G$ of order $pq$ with presentation (\ref{pqpres}).  We first state the basic facts we need about $G$, the proofs of which are trivial.

\begin{cor}\label{pqbasics}
    Let $G$ and its presentation be as above.  Then in the notation of Section \ref{kqsect},
    \begin{enumerate}
      \item $c=|Z(G)|=1$
      \item $d\equiv 0\bmod k$
      \item $G$ has $q$ non-isomorphic 1-dimensional representations, given by Corollary \ref{kqchar}, and $\frac{p-1}{q}$ non-isomorphic irreducible $q$-dimensional representations, given by Lemma \ref{kqnonlinchar}.  These give all of the non-isomorphic irreducible representations of $G$.
    \end{enumerate}
\end{cor}

We now state the indicators for the irreducible $G$-modules.

\begin{thm}
  Let $G$ be a non-abelian group of order $pq$ as above.  Let $V$ be a 1-dimensional $G$-module with character $\chi=\xi_s$ as in Theorem \ref{kq1dim}.  Then
  $$\nu_m(V) = \left\{ \begin{array}{cll} 1 &;& q\mid ms\\
  0&;& q\nmid ms \end{array}\right..$$
\end{thm}
\begin{proof}
  This follows from Corollary \ref{pqbasics} and Theorem \ref{kq1dim}.
\end{proof}

\begin{thm}
  Let $G$ be a non-abelian group of order $pq$ as above.  Let $V$ be an irreducible q-dimensional $G$-module with character $\phi_j^G$, with notation as in Lemma \ref{kqnonlinchar}.  Then
  $$\nu_m(V) = \left\{ \begin{array}{cll} q &;&pq\mid m\\
  q-1 &;&q\mid m\wedge p\nmid m\\
  1&;&q\nmid m \wedge p\mid m\\
  0&;&q\nmid m\wedge p\nmid m \end{array}\right..$$
\end{thm}
\begin{proof}
  This follows immediately from Corollary \ref{pqbasics} and Theorem \ref{kqnonlinear}.
\end{proof}

We can also explicitly state the indicators of irreducible modules over the double $\D(G)$.

\begin{thm}
  Let $G$ be a non-abelian group of order $pq$ as above, and let $i\in\BZ$ with $p\nmid i$.  Let $V=M(\class(a^i),\phi_j)$ be the irreducible $\D(G)$-module given by Proposition \ref{MOrho}, where $\phi_j$ is the irreducible character of $\cyc{a}$ given by $\phi_j(a)=\mu_p^j$, with $\mu_p$ a fixed primitive $p$-th root of unity.  Then
  $$\nu_m(V) = \left\{ \begin{array}{cll} q &;& pq\mid m\\
  q-1&;& q\mid m\wedge p\nmid m\\
  1 &;& q\nmid m\wedge p\mid m\\
  0&;& q\nmid m\wedge p\nmid m \end{array}\right..$$
\end{thm}
\begin{proof}
  This follows from Corollary \ref{pqbasics} and Theorem \ref{kqdouba}.
\end{proof}

\begin{thm}
  Let $G$ be a non-abelian group of order $pq$ as above and let $q\nmid j$.  Consider the irreducible $\D(G)$-module $V=M(\class(a^i b^j),\chi)$ given in Proposition \ref{MOrho}, where $\chi$ is any irreducible representation (and thus linear character) of $C_{G}(a^i b^j)$.  Then
  $$\nu_m(V) = \left\{ \begin{array}{cll} p &;& pq\mid m\\
  p-\frac{2(p-1)}{q}&;& p\nmid m \wedge q\mid m\\
  0&;& q\nmid m \end{array}\right..$$
\end{thm}
\begin{proof}
  All claims follow from Lemma \ref{kqnormalize}, Corollary \ref{pqbasics}, and Corollary \ref{kqdoubb0}.
\end{proof}

This completes the determination of all indicators for irreducible modules over a non-abelian group of order $pq$ and its double.

\subsection{Semidihedral groups}\label{semiex}
As our last major class of examples, we consider the semidihedral 2-groups \cite{G}.  Specifically:

\begin{df}\label{sdpres}
For any $3\leq N\in\BN$, define the semidihedral 2-group $G$ by
\begin{eqnarray*}
    G = \cyc{a,b \ | \ a^{2^N}=b^2 = 1, bab = a^{2^{N-1}-1}}.
\end{eqnarray*}
\end{df}
Any such group satisfies Definition \ref{presentation}, with $k=2^N$, $q=2$, $n=2^{N-1}-1$, and so we may apply our results to it.  We fix for now a group having the above presentation.

We first state the basic facts we need about $G$, the proofs of which are trivial.

\begin{cor}\label{sdbasics}
    Let $G$ be a semidihedral 2-group as in Definition \ref{sdpres}.  Then in the notation of Sections \ref{kqsect} and \ref{kqdoubsect},
    \begin{enumerate}
      \item $c=|Z(G)|=2$
      \item $d\equiv 2^{N-1}\equiv \frac{k}{c}\bmod 2^N$
      \item $G$ has four non-isomorphic 1-dimensional representations, given by Corollary \ref{kqchar}, and $2^{N-1}-1$ non-isomorphic irreducible $2$-dimensional representations, given by Lemma \ref{kqnonlinchar}.  These give all of the non-isomorphic irreducible representations of $G$.
    \end{enumerate}
\end{cor}

We now state the indicators for the irreducible $G$-modules.

\begin{thm}
  Let $G$ be a semidihedral 2-group as in Definition \ref{sdpres}.  Let $V$ be a 1-dimensional $G$-module with character $\chi=\xi_s\otimes \psi_t\in\widehat{\BZ_2}\otimes\widehat{\BZ_2}$ as in Theorem \ref{kq1dim}.  Then
  $$\nu_m(V) = \left\{ \begin{array}{cll} 1 &;& 2\mid ms \wedge 2\mid mt\\
  0&;& 2\nmid ms \vee 2\nmid mt\end{array}\right..$$
\end{thm}
\begin{proof}
  This follows from Corollary \ref{sdbasics} and Theorem \ref{kq1dim}.
\end{proof}

\begin{thm}
  Let $G$ be a semidihedral 2-group as in Definition \ref{sdpres}.  Let $V$ be an irreducible $2$-dimensional $G$-module with character $\phi_j^G$, with notation as in Lemma \ref{kqnonlinchar}.  Then
  $$\nu_m(V) = \left\{ \begin{array}{cll} 2 &;&2\mid m \wedge 2^N \mid mj\\
  1&;&2\mid m \wedge 2^N\nmid mj \wedge 4\mid mj\\
  0&;&2\nmid m \vee (2\mid m \wedge 4\nmid mj)\end{array}\right..$$
\end{thm}
\begin{proof}
  This follows immediately from $N\geq 3$, the fact that $V$ must be irreducible, Corollary \ref{sdbasics}, and Theorem \ref{kqnonlinear}.
\end{proof}

\begin{thm}
  Let $G$ be a semidihedral 2-group as in Definition \ref{sdpres}, and let $i\in\BZ$ $i\not\equiv 0\bmod 2^{N-1}$.  Let $V=M(\class(a^i),\phi_j)$ be the irreducible $\D(G)$-module given by Proposition \ref{MOrho}, where $\phi_j$ is the irreducible character of $\cyc{a}$ given by $\phi_j(a)=\mu^j$, with $\mu$ a fixed primitive $2^N$-th root of unity.  Then
  $$\nu_m(V) = \left\{ \begin{array}{cll} 2 &;& 2\mid m \wedge 2^N\mid mi \wedge 2^N\mid mj\\
  1 &;& 2\mid m \wedge 4\mid mi\wedge 4\mid mj\wedge (2^N\nmid mi\vee 2^N\nmid mj)\\
  0&;& (2\nmid m \vee 4\nmid mi \vee 4\nmid mj) \wedge (2^N\nmid mi \vee 2^N\nmid mj) \end{array}\right..$$
\end{thm}
\begin{proof}
  This follows from $N\geq 3$, assumptions on $i$ and $j$, Corollary \ref{sdbasics} and Theorem \ref{kqdouba}.
\end{proof}

\begin{thm}
  Let $G$ be a semidihedral 2-group as in Definition \ref{sdpres}, and let $i\in\BN$.
  Consider the irreducible $\D(G)$-module $V=M(\class(a^i b),\chi)$ given in Proposition \ref{MOrho},
  where $\chi$ is any irreducible representation (and thus linear character) of $C_{G}(a^i b)$.  Then
  \renewcommand\arraystretch{2.25} 
    \begin{eqnarray*}
        \nu_m(V) = \left\{ \begin{array}{cll}
            \displaystyle{\frac{\gcd(m,2^N)}{2}}&;& 2\mid m \wedge \left( 4\mid m \vee \chi|_{Z(G)}=1\vee 2\mid i\right)\\
            -1&;& 2\mid m\wedge 4\nmid m \wedge \chi|_{Z(G)}\neq 1 \wedge 2\nmid i \\
            0&;& \mbox{otherwise}
        \end{array}\right..
  \end{eqnarray*}
\end{thm}
\begin{proof}
  All claims follow from Lemma \ref{kqnormalize}, Corollary \ref{sdbasics}, and Theorem \ref{kqdoubb}.
\end{proof}

This completes the determination of all indicators for irreducible modules over a semidihedral 2-group and its double.

\section{The Generalized Quaternion Groups $Q_{4n}$}\label{quatsect}

We wish to conclude now by computing the indicators for another family of groups and their doubles.  In particular, we consider the generalized quaternion groups.
\begin{df}\label{quatpres}
For $2\leq n\in\BN$, the generalized quaternion group with $4n$ elements is
\begin{eqnarray*}
  Q_{4n} = \cyc{a,b \ | \ a^{2n}=1, b^2 = a^{n}, bab^{-1} = a^{-1}} && n\geq 2.
\end{eqnarray*}
\end{df}
This presentation is very similar to that of the dihedral groups, and much of what we have done in the previous sections for $G\cong \G$ (with $l=1$ in particular) will apply to the quaternion groups with only minor changes.  Furthermore, the case when $n$ is a power of 2 gives the remaining isomorphism classes of non-abelian $2$-groups $G$ with $[G:G']=4$ \cite[Theorem III 11.9(a)]{Hu}.

We begin, as before, by considering the structure of the group $Q_{4n}$.

\begin{lem}\label{quatident}
  Let $Q_{4n}$ be as in Definition \ref{quatpres}.  Then for any $i,j\in\BZ$ the following identities hold:
  \begin{enumerate}
    \item $b a^i = a^{-i} b$
    \item If $2\mid j$, then $(a^i b)^j = b^j$.
    \item If $2\nmid j$, then $(a^i b)^j = a^i b^j$.
  \end{enumerate}
\end{lem}
\begin{proof}
  These are all easy consequences of the presentation.
\end{proof}
This readily establishes what the center of $Q_{4n}$ is:
\begin{cor}\label{quatcent}
    Let $Q_{4n}$ be a generalized quaternion group with presentation (\ref{quatpres}).  Then
    $$Z(G) = \cyc{a^{n}}\cong \BZ_2.$$
\end{cor}
We also have the analogue of Lemma \ref{kqclass} for $Q_{4n}$.
\begin{lem}\label{quatclass}
    Let $Q_{4n}$ be a generalized quaternion group with presentation (\ref{quatpres}) and let $i\in\BZ$.
    \begin{enumerate}
      \item If $a^i\not\in Z(G)$, then $\class(a^i) = \{a^i, a^{-i}\}$ and $|\class(a^i)|=2$.  In particular, there are $2n-1$ distinct conjugacy classes of non-central powers of $a$.
      \item $\class(a^i b) = \{ a^{i-2s}b\}_{s=0}^{n-1}$ and $|\class(a^i b)|= n$.  Thus there are two distinct conjugacy classes of this form, depending on whether $i$ is even or odd.
      \item All conjugacy classes in $Q_{4n}$ are either singletons or one of the above.
    \end{enumerate}
\end{lem}
\begin{proof}
  iii) Follows from the class equation once we have established the first two parts.

  For i), we have for any $s\in\BZ$ that
  $$(a^s b)^{-1} a^i (a^s b) = b^{-1} a^{-s} a^i a^s b = b^{-1} a^i b = a^{-i}.$$
  Since $a^s$ clearly centralizes $a^i$, we conclude that $\class(a^i) = \{a^i, a^{-i}\}$.  The remaining claims then follow.

  For ii), for any $s,t\in\BZ$ we have
  $$a^{-s} a^i b a^{s} = a^{i-2s}b$$
  and
  \begin{eqnarray*}
    (a^t b)^{-1} a^i b (a^t b) &=& b^{-1} a^{-t} a^i b a^t b = a^{2t-i}b.
  \end{eqnarray*}
  Taking $t = i-s$, we see that each of these relations gives us the same set of conjugates.  We conclude that $\class(a^i b) = \{ a^{i-2s}b\}_{s=0}^{n-1}$.  The remaining claims then follow.
\end{proof}

The last result we wish to establish about the structure of $Q_{4n}$ concerns the centralizers of its elements.  This is the analogue of Lemma \ref{kqnormalize}.

\begin{lem}\label{quatnorm}
  Let $Q_{4n}$ be as in Definition \ref{quatpres} and suppose $i,j\in\BZ$.  Then
  \begin{enumerate}
    \item $C_G(g) = G$ $\forall g\in Z(G)$

    \item $C_G(a^i) = \cyc{a}\cong \BZ_{2n} \iff a^i\not\in Z(G)$

    \item $C_G(a^i b) = \cyc{a^i b} \cong \BZ_4$
  \end{enumerate}
\end{lem}
\begin{proof}
  i) Is trivial

  ii) and iii) Follow from Lemma \ref{quatident} and the proof of Lemma \ref{quatclass}.
\end{proof}

Again as before, our next task is to determine the representation theory of $Q_{4n}$.

\begin{prop}\label{quatchar}
  $Q_{4n}$ has four irreducible $1$-dimensional representations.  They arise from the tensor products of the irreducible characters of $Z(Q_{4n})\cong \BZ_2$ and $Q_{4n}/\cyc{a} \cong \BZ_2$.

  For $H=\cyc{a}\,\lhd\,Q_{4n}$, consider the irreducible character $\phi_j$ of $H$ given by $\phi_j(a) = \mu_{2n}^j$, where $\mu_{2n}$ is a primitive $2n$-th root of unity.  Then $\phi_j^{Q_{4n}}$, the induced character on $Q_{4n}$, is of dimension $2$.  It is irreducible $\iff$ $j\not\equiv 0\bmod n$, and there are $n-1$ distinct isomorphism classes amongst these characters.

  Finally, any irreducible representation of $Q_{4n}$ is of one of the above forms.
\end{prop}
\begin{proof}
  The proof of this is exactly the same as the relevant proofs for groups of the form $\G$ (with $l=1$) from Section \ref{kqrepsect}.
\end{proof}

\section{Indicators for $Q_{4n}$ and $\D(Q_{4n})$}\label{quatindsect}
Now that we have determined all of the basic structure and representation theory of $Q_{4n}$, we can proceed to calculate its indicators and the indicators of its double.

\begin{thm}\label{quat1dim}
  Let $Q_{4n}$ be as in Definition \ref{quatpres}.  Let $V$ be a 1-dimensional $Q_{4n}$-module with character $\chi=\xi\otimes \psi\in \widehat{\BZ_2}\otimes\widehat{\BZ_2}$.  Then for any $m\in\BN$ we have
  \begin{eqnarray*}
    \nu_m &=& \left\{ \begin{array}{cll}
        1 &;& \xi^m = 1 \wedge \psi^m = 1\\
        0 &;& \xi^m\neq 1 \vee \psi^m\neq 1
        \end{array}\right.
  \end{eqnarray*}
\end{thm}
\begin{proof}
  That every $1$-dimensional character of $Q_{4n}$ has the indicated form follows immediately from Proposition \ref{quatchar}.  The proof is then identical to that of Theorem \ref{kq1dim}.
\end{proof}

\begin{thm}\label{quatnonlinear}
  Let $Q_{4n}$ be as in Definition \ref{quatpres}.  Let $V_j$ be an irreducible $2$-dimensional $Q_{4n}$-module with character $\phi_j^{Q_{4n}}$ as in Lemma \ref{quatchar}.  Then
  \begin{eqnarray*}
    \nu_m(V_j) &=& \left\{ \begin{array}{rll}
    2&;&2\mid m\wedge 2n\mid mj \wedge 4\mid mj\\
    1&;&\left(2\nmid m \wedge 2n\mid mj\right)\vee \left(2\mid m \wedge 4\mid mj \wedge 2n\nmid mj\right)\\
    0&;&\left(2\nmid m \wedge 2n\nmid mj\right) \vee \left( 2\mid m \wedge 4\nmid mj \wedge 2n\mid mj\right)\\
    -1&;&2\mid m \wedge 2n\nmid mj \wedge 4\nmid mj \end{array}\right.
  \end{eqnarray*}
\end{thm}
\begin{proof}
  The case $2\nmid m$ is the same as it was in the proof of Theorem \ref{kqnonlinear}, \it{mutatis mutandi}.

  If $2\mid m$, then using Lemma \ref{quatident} and the fact that $\phi_j(a^{n}) = \phi_1^j(a^{n})=(-1)^j$, we get
  \begin{eqnarray*}
    \nu_m(V_j) &=& \frac{1}{4n} \sum_{g\in Q_{4n}}\phi_j^{Q_{4n}}(g^m)\\
    &=& \frac{1}{4n} \sum_{g\in\cyc{a}} \left( \phi_j^G(g^m) + \phi_j^G(b^m)\right)\\
    &=& \ip{\phi_{jm}}{1} + \phi_j(a^{mn/2})\\
    &=& \left\{ \begin{array}{rll}
        2 &;& 2n\mid mj\wedge 4\mid mj\\
        1 &;& 2n\nmid mj \wedge 4\mid mj\\
        0 &;& 2n\mid mj \wedge 4\nmid mj\\
        -1 &;& 2n\nmid mj \wedge 4\nmid mj
        \end{array}\right.
  \end{eqnarray*}
  Combining these cases, we get the desired formula for $\nu_m(V)$.
\end{proof}
In particular, we get the well-known fact that $Q_{4n}$ has irreducible modules which admit an invariant, non-degenerate, skew-symmetric bilinear form.  Indeed, it is readily checked that $\nu_2(V_j) = -1$, with $V_j$ as above, precisely when $j$ is odd.

We now wish to determine the indicators for the irreducible $\D(Q_{4n})$-modules.  As before, we will need to compute the sets in Definition \ref{Gmsetdef} for $G=Q_{4n}$.

\begin{prop}\label{quatgmsets}
    Let $G=Q_{4n}$ be as in Definition \ref{quatpres}.  Let $i,s\in\BZ$ and $m\in\BN$.  Consider the sets $G_m(x)$, $x\in G$, as given in Definition \ref{Gmsetdef}.  Then the following hold:
    \begin{enumerate}
      \item $a^s\in G_m(a^i) \iff 2n\mid mi$
      \item $a^s b\in G_m(a^i) \iff 2\mid m$
      \item $a^s \in G_m(a^i b) \iff 2\mid m \wedge \big(\left( 4\mid m \wedge 2n\mid ms\right) \vee \left(4\nmid m \wedge 2n\mid n+ms\right)\big)$
      \item $a^s b\in G_m(a^i b) \iff 2\mid m \wedge \big( \left( 4\mid m \wedge 2n\mid m(i-s)\right) \vee \left(4\nmid m \wedge 2n\mid n+m(i-s)\right)\big)$
    \end{enumerate}
\end{prop}
\begin{proof}
  For i), we have
  $$\prod_{j=0}^{m-1} a^{-sj}a^i a^{sj} = a^{mi} = 1 \iff 2n\mid mi.$$

 For the remainder of the proof, for any $j\in\BZ$ we let $\delta_j$ be defined to be $1$ when $j$ is odd, and 0 otherwise.

  For ii),
  \begin{eqnarray*}
    \prod_{j=0}^{m-1}(a^s b)^{-j} a^i (a^s b)^j &=& \prod_{j=0}^{m-1} b^{-j} a^{-\delta_j s} a^i a^{\delta_j s} b^j\\
    &=& \prod_{j=0}^{m-1} b^{-j} a^i b^j\\
    &=& \prod_{j=0} a^{(-1)^ji}\\
    &=& a^{i\sum_{j=0}^{m-1} (-1)^j}.
  \end{eqnarray*}
  Whence, $a^s b\in G_m(a^i)$ $\iff$ $2\mid m$.

  Similarly, for iii) we have
  \begin{eqnarray*}
    \prod_{j=0}^{m-1} a^{-sj}a^i b a^{sj} &=& \prod_{j=0}^{m-1} a^{i-2sj}b\\
    &=& a^{\sum_{j=0}^{m-1} (-1)^j(i-2sj)}b^m.
  \end{eqnarray*}
  So for this to equal 1, we must have that $2\mid m$, since then $b^m\in Z(G)\subseteq \cyc{a}$.  So if $2\mid m$, we have
  $$a^{\sum_{j=0}^{m-1} (-1)^j(i-2sj)}b^m = a^{ms} b^m.$$
  Since $b^4=1$ and $b^2=a^{n}$, we conclude that
  $$a^s\in G_m(a^i b)\iff 2\mid m \wedge \big(\left( 4\mid m \wedge 2n\mid ms\right) \vee \left(4\nmid m \wedge 2n\mid n+ms\right)\big).$$

  Finally, for iv) we have
  \begin{eqnarray*}
    \prod_{j=0}^{m-1} (a^s b)^{-j} a^i b (a^s b)^j &=& \prod_{j=0}^{m-1} b^{-j} a^{-\delta_j s}a^i b a^{\delta_j s}b^j\\
    &=& \prod_{j=0}^{m-1} a^{(-1)^j(i-2\delta_j s)}b\\
    &=& a^{\sum_{j=0}^{m-1}(i-2\delta_j s)} b^m.
  \end{eqnarray*}
  Again we conclude that for this latter element to be the identity we must have $2\mid m$.  So assuming $2\mid m$, we have
  $$a^{\sum_{j=0}^{m-1}(i-2\delta_j s)} b^m = a^{m(i-s)}b^m,$$
  and we conclude, similarly to before, that
  $$a^s b\in G_m(a^i b) \iff 2\mid m \wedge \big( \left( 4\mid m \wedge 2n\mid m(i-s)\right) \vee \left(4\nmid m \wedge 2n\mid n+m(i-s)\right)\big).$$
\end{proof}

Equipped with this result, we can now compute the indicators for the irreducible $\D(G)$-modules.

\begin{thm}\label{quatdouba}
  Let $Q_{4n}$ be a generalized quaternion group, with presentation (\ref{quatpres}).  Suppose $a^i\not\in Z(G)$ and let $\chi_j$ be the irreducible character of $\cyc{a}$ given by $\chi_j(a)=\mu_{4n}^j$, where $\mu_{4n}$ is any fixed primitive $4n$-th root of unity.  Then the irreducible $\D(Q_{4n})$-module $V=M(\class(a^i),\chi_j)$ from Proposition \ref{MOrho} has the following indicators:
  \begin{eqnarray*}
    \nu_m(V) = \left\{ \begin{array}{rll}
        2&;&2\mid m\wedge 4\mid mj \wedge 2n\mid mi \wedge 2n \mid mj\\
        1&;& \left(2\nmid m \wedge 2n\mid mi \wedge 2n\mid mj\right) \vee \left( 2\mid m \wedge 4\mid mj \wedge (2n\nmid mi \vee 2n\nmid mj)\right)\\
        0&;& \left(2\nmid m \wedge (2n\nmid mi \vee 2n\nmid mj)\right) \vee \left( 2\mid m \wedge 4\nmid mj \wedge 2n\mid mi \wedge 2n\mid mj\right)\\
        -1&;& 2\mid m \wedge 4\nmid mj \wedge \left(2n\nmid mi \vee 2n\nmid mj\right)
        \end{array}\right.
  \end{eqnarray*}
\end{thm}
\begin{proof}
  By Proposition \ref{quatgmsets}, taking $G=Q_{4n}$, we have
  \begin{eqnarray*}
    a^s\in G_m(a^i)&\iff& 2n\mid mi\\
    a^sb\in G_m(a^i)&\iff& 2\mid m.
  \end{eqnarray*}

  Subsequently,
  \begin{eqnarray}\label{quatdouba1}
    \frac{1}{4n} \sum_{g\in\class(a^i)} \sum_{s=0}^{4n-1} \Tr_V(p_g\# a^{ms}) &=& \left\{ \begin{array}{cll}
        \frac{1}{2n}\sum_{s=0}^{4n-1}\chi_j(a^{ms})&;&2n\mid mi\\
        0&;&2n\nmid mi
        \end{array}\right.\\
    &=& \left\{ \begin{array}{cll}
        \ip{\chi_{mj}}{1}&;& 2n\mid mi\\
        0&;& 2n\nmid mi \end{array}\right.\nonumber\\
    &=& \left\{ \begin{array}{cll}
        1&;&2n\mid mi \wedge 2n\mid mj\\
        0&;& 2n\nmid mi\vee 2n\nmid mj \end{array}\right.\nonumber
  \end{eqnarray}

  Additionally,
  \begin{eqnarray}\label{quatdouba2}
    \frac{1}{4n} \sum_{g\in\class(a^i)} \sum_{s=0}^{4n-1} \Tr_V(p_g\# (a^s b)^m) &=& \left\{ \begin{array}{cll}
        \frac{1}{2n} \sum_{s=0}^{4n-1} \chi_j(b^m) &;& 2\mid m\\
        0&;&2\nmid m \end{array}\right.\\
    &=& \left\{ \begin{array}{cll}
        1&;&4\mid m\\
        \chi_j(a^{n}) &;& 2\mid m\wedge 4\nmid m\\
        0&;& 2\nmid m \end{array}\right.\nonumber\\
    &=& \left\{ \begin{array}{rll}
        1 &;& 2\mid m \wedge 4\mid mj\\
        0 &;& 2\nmid m\\
        -1&;& 2\mid m \wedge 4\nmid mj \end{array}\right.\nonumber
  \end{eqnarray}
  Using that $\nu_m(V) = (\ref{quatdouba1}) + (\ref{quatdouba2})$, and simplifying as appropriate, we get the desired formula.
\end{proof}
We make the quick remark that many of the cases above vanish or otherwise simplify whenever $n$ is even, since then the conditions $2n\mid mj$ and $4\nmid mj$ cannot occur simultaneously.

We have one major remaining class of irreducible $\D(Q_{4n})$-modules to consider now.

\begin{thm}\label{quantdoubb}
  Let $G\cong Q_{4n}$ be a generalized quaternion group with presentation (\ref{quatpres}), and let $i,m\in\BN$.  Suppose $\chi$ is an irreducible character of $C_G(a^i b)$.  Then the irreducible $\D(Q_{4n})$-module $V=M(\class(a^i b),\chi)$ has the following indicators:
  \begin{eqnarray*}
    \nu_m(V) &=& \left\{ \begin{array}{cll}
        \gcd(\frac{m}{2},n)&;&4\mid m\\
        \chi(a^{n})\gcd(\frac{m}{2},n)&;&2\mid m \wedge 4\nmid m\\
        0&;& 2\nmid m
        \end{array} \right.
  \end{eqnarray*}
\end{thm}
\begin{proof}
  By Proposition \ref{quatgmsets}, we have that
  \begin{eqnarray*}
    a^s\in G_m(a^i b)&\iff& 2\mid m \wedge \big(\left( 4\mid m \wedge 2n\mid ms\right) \vee \left(4\nmid m \wedge 2n\mid n+ms\right)\big)\\
    a^sb \in G_m(a^i b) &\iff& 2\mid m \wedge \big( \left( 4\mid m \wedge 2n\mid m(i-s)\right) \vee \left(4\nmid m \wedge 2n\mid n+m(i-s)\right)\big).
  \end{eqnarray*}
  In particular, if $2\nmid m$, then $\nu_m(V) = 0$.  So in the remainder of the proof we assume that $2\mid m$.

  To determine the contribution of the $a^s$ terms to $\nu_m(V)$, we first observe that $\left| \{ g\in \cyc{a} \ | \ g\in G_m(a^i b)\} \right|=\gcd(m,2n).$  Subsequently,
  \begin{eqnarray}\label{kqdoubb1}
    \sum_{a^s\in G_m(a^i b)} \chi(a^{ms}) &=& \left\{ \begin{array}{cll}
        \gcd(m,2n) &;& 4\mid m\\
        \chi(a^{n})\gcd(m,2n)&;& 4\nmid m \end{array}\right.
  \end{eqnarray}

  For the contribution of the $a^s b$ terms to $\nu_m(V)$, we observe that $\left| \{ a^s b\in G_m(a^ib) \} \right|=\gcd(m,2n)$.  Subsequently,
  \begin{eqnarray}\label{kqdoubb2}
    \ \sum_{a^s b\in G_m(a^i b)} \chi((a^s b)^m) = \sum_{a^s b\in G_m(a^i b)} \chi(b^m) = \left\{ \begin{array}{cll}
        \gcd(m,2n)&;& 4\mid m\\
        \chi(a^{n})\gcd(m,2n)&;& 4\nmid m
    \end{array}\right.
  \end{eqnarray}

  The desired formula follows from the fact
  $$\nu_m(V) = \frac{1}{4}\left( (\ref{kqdoubb1}) + (\ref{kqdoubb2}) \right).$$
\end{proof}

Summarizing, we have
\begin{thm}\label{quatsum}
  Let $G\cong Q_{4n}$ be as in Definition \ref{quatpres}, and let $H=\D(G)$.  Suppose $V$ is a $G$-module and $W$ is an $H$-module.  Then for any $m\in\BN$, we have $\nu_m(V),\nu_m(W)\in\BZ$.
\end{thm}

\section{Closing Remarks and Questions}
We wish to end this paper by pointing out a few questions that have arisen which we do not have answers for. In many of our questions it should be noted that we are implicitly also asking if the indicators for the relevant doubles are actually real numbers or even all integers.  Whether or not this is true in general remains an open question.

Our first pair of questions concern the parameter $d$ from equation (\ref{dvalue}).  The proof (or disproof) of these would likely involve prime factorization results for values obtained by evaluating cyclotomic polynomials at certain values (see \cite{M1} \cite{M2} \cite{W} for results in this vein).
\begin{question}
  Let $k,q,n$ be as in Definition \ref{presentation}, and define $c,d$ as in equations (\ref{cvalue}) and (\ref{dvalue}) respectively.  Set $h=\gcd(d,k)$.  In Example \ref{splitex}, it was noted that we knew of no examples with $q>2$, $q\nmid \displaystyle{\frac{k}{h}}$, and $\gcd(c,k/c)\neq 1$.  Does $q>2$ and $q\nmid \displaystyle{\frac{k}{h}}$ force $\gcd(c,k/c)=1$?
\end{question}
It is known \cite{GM} that, in general, $G$ totally orthogonal need not imply that $\D(G)$ is totally orthogonal.  The example there gives a totally orthogonal group $G$ and an irreducible $\D(G)$-module $V$ with $\nu_2(V)=0$.  \cite{GM} also shows that real reflection groups $G$ do enjoy the property that $\D(G)$ is totally orthogonal.

Our results, in particular Corollary \ref{negind}, additionally show that an irreducible $\D(G)$-module $V$, may have $\nu_m(V)<0$ but $\nu_2(V)\geq 0$.  In our situation, $G$ is not totally orthogonal by Theorem \ref{kqortsummary}.  Given this, we posit the following series of questions/tasks:
\begin{question}\label{bothort}
  What are some (non-trivial) necessary or sufficient conditions for both $G$ and $\D(G)$ to be totally orthogonal?  What is a classification of all such groups?  In particular, is there an example where $G$ is not a real reflection group?
\end{question}
\begin{question}
    If $G$ is a finite real reflection group and $V$ is an irreducible $\D(G)$-module, then is $\nu_m(V)\geq 0$ for all $m$?

    More generally, if $G$ is a totally orthogonal finite group such that every irreducible $\D(G)$-module $V$ satisfies $\nu_2(V)\geq 0$, then is it also true that $\nu_m(V)\geq 0$ for all $m$ and every such $V$?  What if we suppose, instead, that $\D(G)$ is also totally orthogonal?
\end{question}
Our results in Section \ref{dnexample} show the answer to the first part is true in the special case of the dihedral groups.  Preliminary calculations by the author show that it is also true for $S_4$.  We subsequently conjecture that it is true for all finite real reflection groups.
\begin{question}
  Does there exist a totally orthogonal finite group $G$ such that $\D(G)$ is not totally orthogonal, but $|\,\nu_2(V)|=1$ for all irreducible $\D(G)$-modules $V$?  Note that, by Lemma \ref{FxF}, any such group would necessarily have the property that $Z(G)\cong \BZ_2^r$ for some $r\geq0$; this also applies to Question \ref{bothort}.  One can also ask if there is a finite group $G$ such that $|\,\nu_2(V)|=1$ for every irreducible $G$-module and $\D(G)$-module.
\end{question}

\section*{Acknowledgements}

This paper was written as part of the author's Ph.D. thesis at the University of Southern California.  The author thanks S. Montgomery for her frequent advice and guidance.  The author also thanks R. Guralnick and S. Ng for their helpful suggestions and conversations.


\begin{thebibliography}{FGSV}
\bibitem[AF]{AF} N. Andruskiewitsch; F. Fantino, New techniques for pointed Hopf algebras. \it{New developments in
Lie theory and geometry}, 323--348, Contemp. Math., 491, \it{Amer. Math. Soc.}, \it{Providence, RI}, 2009.

\bibitem[B1]{B1}  P. Bantay, The Frobenius-Schur indicator in conformal field theory, \it{Physics Lett.
B} 394 (1997), no. 1-2, 87-88.

\bibitem[B2]{B2}  P. Bantay, Frobenius-Schur indicators, the Klein-bottle amplitude, and the principle
of orbifold covariance, \it{Phys. Lett. B} 488 (2000), 207-210.

\bibitem[FGSV]{FGSV}  J. Fuchs, A. Ch. Ganchev, K. Szlach\'{a}nyi and P.
Vecserny\'{e}s, $S_{4}$ symmetry of $6j$ symbols and Frobenius-Schur
indicators in rigid monoidal $C^{\ast }$ categories. \it{J. Math.
Phys.} 40 (1999), no. 1, 408--426.

\bibitem[G]{G} D. Gorenstein, Finite Groups, Harper's Series in Modern Mathematics, Harper \& Row, Publishers, New York, Evanston, and London (1968)

\bibitem[GW]{GW} L. C. Grove and K. S. Wang, Realizability of representations of finite
groups, \it{J. Pure Appl. Algebra} 54 (1988) 299--310.

\bibitem[GM]{GM} R. Guralnick and S. Montgomery, Frobenius-Schur Indicators for subgroups
and the Drinfel'd double of Weyl groups, \it{Trans. Amer. Math. Soc.} 361 (2009), no. 7, 3611--3632.

\bibitem[Hu]{Hu} B. Huppert, \it{Endliche Gruppen I}, Springer-Verlag, Berlin, 1967

\bibitem[I]{I} I. Martin Isaacs, \it{Character Theory of Finite Groups}, Dover Publications, Inc., Mineola, NY, 1994

\bibitem[JM]{JM} A. Jedwab and S. Montgomery, Representations of Hopf Algebras associated to $S_n$, \it{Algebr. Represent. Theory} 12 (2009), no. 1, 1--17.

\bibitem[K]{K} Y. Kashina, On semisimple Hopf algebras of dimension $2^m$, \it{Algebras and Representation Theory} 6 (2003), no.4, 393-425

\bibitem[KMM]{KMM} Y. Kashina, G. Mason and S. Montgomery, Computing the Frobenius-Schur
indicator for abelian extensions of Hopf algebras, \it{J. Algebra} 251 (2002), 888--913.

\bibitem[KSZ1]{KSZ1} Y. Kashina, Y. Sommerh\"auser, and Y. Zhu, Self-dual modules
of semisimple Hopf algebras,  \it{J. Algebra} 257 (2002), 88--96.

\bibitem[KSZ2]{KSZ2} Y. Kashina, Y. Sommerh\"auser, and Y. Zhu, On higher Frobenius-Schur
indicators, \it{AMS Memoirs} 181 (2006), no. 855.

\bibitem[LR1]{LR1} R. Larson and D. Radford, Semisimple cosemisimple
Hopf algebras, \it{Am. J. Math.} 109 (1987), 187--195.

\bibitem[LR2]{LR2} R. Larson and D. Radford, Finite-dimensional cosemisimple
Hopf algebras in characteristic 0 are semisimple, \it{J. Algebra} 117 (1988), 267--289.

\bibitem[LMS]{LMS} R. Landers, S. Montgomery and P. Schauenburg,
Hopf powers and orders of some bismash products, \it{J. Pure and
Applied Algebra} 205 (2006), 156--188.

\bibitem[LM]{LM}  V. Linchenko and S. Montgomery, A Frobenius-Schur theorem
for Hopf algebras, \it{Algebras and Representation Theory}, 3
(2000), 347--355.

\bibitem[MaN]{MaN} G. Mason and S-H. Ng, Central invariants and Frobenius-Schur indicators
for semisimple quasi-Hopf algebras, \it{Advances in Math} 190 (2005), 161-195.

\bibitem[Mo]{Mo}  S. Montgomery, \it{Hopf Algebras and their Actions on Rings},
CBMS Lectures, Vol. 82, AMS, Providence, RI, 1993.

\bibitem[M1]{M1} K. Motose, On values of cyclotomic polynomials. \it{International Symposium on Ring Theory (Kyongju, 1999)}, 231--234, Trends Math., \it{Birkh\"auser Boston, Boston, MA}, 2001.

\bibitem[M2]{M2} K. Motose, On values of cyclotomic polynomials. V. \it{Math. J. Okayama Univ.} 45 (2003),
29--36.

\bibitem[N1]{N1}  S. Natale, On group-theoretical Hopf algebras and exact factorizations of
finite groups,  \it{J. Algebra} 270 (2003), 199-211.

\bibitem[N2]{N2}  S. Natale, Frobenius-Schur indicators for a class of fusion categories,
\it{Pacific J. Math} 221 (2005), 363--377.

\bibitem[NS1]{NS1} S-H. Ng and P. Schauenburg, Central invariants and higher indicators for semisimple
Quasi-Hopf algebras, Trans. Amer. Math. Soc. 360 (2008), 1839-1860.

\bibitem[NS2]{NS2} S-H. Ng and P. Schauenburg, Frobenius-Schur indicators and exponents of spherical
categories, \it{Advances in Math} 211 (2007) no. 1, 34-71.

\bibitem[NS3]{NS3} S-H. Ng and P. Schauenburg, Higher Frobenius-Schur indicators for pivotal categories. \it{Hopf algebras and generalizations}, 63--90, Contemp. Math., 441, \it{Amer. Math. Soc.}, Providence, RI, 2007.

\bibitem[R]{R} J. Rotman, An Introduction to the Theory of Groups, 4th Edition, Springer-Verlag, New York (1999)

\bibitem[S]{S} T. A. Springer, A construction of representations of Weyl groups, \it{Inventiones
Math} 44 (1978), 279-293.

\bibitem[W]{W} H.C. Williams, \'{E}douard Lucas and Primality Testing, A Wiley-Interscience publication, Canadian Mathematical Society series of monographs, 1998.

\end{thebibliography}
\end{document}